\documentclass[11pt]{amsart}
\usepackage{amsmath}
\usepackage{amssymb}
\usepackage{amsfonts}
\usepackage{amsthm}
\usepackage{enumerate}
\usepackage[mathscr]{eucal}
\usepackage{graphicx}
\usepackage[pdftex, pagebackref, bookmarksnumbered, bookmarksopen, colorlinks, citecolor=blue, linkcolor=blue]{hyperref}

\newtheorem{theorem}{Theorem}[section]
\newtheorem{lemma}[theorem]{Lemma}

\theoremstyle{definition}
\newtheorem{definition}[theorem]{Definition}
\newtheorem{prop}[theorem]{Proposition}

\theoremstyle{remark}
\newtheorem{remark}[theorem]{Remark}

\numberwithin{equation}{section}

\def\intslash{\rlap{\kern  .32em $\mspace {.5mu}\backslash$ }\int}
\def\qsl{{\rlap{\kern  .32em $\mspace {.5mu}\backslash$ }\int_{Q_x}}}

\def\R{{\mathbb R}}
\def\Z{{\mathbb Z}}
\def\M{{\mathcal M}}
\def\C{{\mathcal C}}
\def\F{{\mathcal F}}

\def\S{\mathbf S}

\def\N{\mathbb N}

\def\emph#1{{\it #1 }}

\def\pari{\partial}

\def\eg{{\it e.g. }}

\def\supp{{\text{\rm supp}}}

\def\inn#1#2{\langle#1,#2\rangle}

\def\rta{\rightarrow}

\def\card{\text{\rm card}}
\def\lc{\lesssim}

\def\pv{\text{\rm p.v.}}
\def\alp{\alpha}

\def\del{\delta}             
\def\eps{\varepsilon}

\def\lam{\lambda}            \def\Lam{\Lambda}

\def\vphi{\varphi}

\def\fr{\frac}

\newcommand{\Be}{\begin{equation}}
\newcommand{\Ee}{\end{equation}}
\newcommand{\Bes}{\begin{equation*}}
\newcommand{\Ees}{\end{equation*}}
\newcommand{\Bsp}{\begin{split}}
\newcommand{\Esp}{\end{split}}
\newcommand{\Bm}{\begin{multline}}
\newcommand{\Em}{\end{multline}}
\newcommand{\Bea}{\begin{eqnarray}}
\newcommand{\Eea}{\end{eqnarray}}
\newcommand{\Beas}{\begin{eqnarray*}}
\newcommand{\Eeas}{\end{eqnarray*}}
\newcommand{\Benu}{\begin{enumerate}}
\newcommand{\Eenu}{\end{enumerate}}
\newcommand{\Bi}{\begin{itemize}}
\newcommand{\Ei}{\end{itemize}}
\begin{document}

\title[Maximal operator of Calder\'on commutator]{Maximal operator for the higher order Calder\'on commutator}

\author{Xudong Lai}
\address{Xudong Lai: Institute for Advanced Study in Mathematics, Harbin Institute of Technology, Harbin, 150001, People's Republic of China}
\email{xudonglai@hit.edu.cn\ xudonglai@mail.bnu.edu.cn}
\thanks{This work was supported by China Postdoctoral Science Foundation (No. 2017M621253, No. 2018T110279), the National Natural Science Foundation of China (No. 11801118) and the Fundamental Research Funds for the Central Universities.}





\begin{abstract}
In this paper, we investigate the weighted multilinear boundedness properties of the maximal higher order Calder\'on commutator for the dimensions larger than two.
We establish all weighted multilinear estimates on the product of the $L^p(\R^d,w)$ space, including  some peculiar endpoint estimates of the higher dimensional Calder\'on commutator.
\end{abstract}

\maketitle

\section{Introduction}

In the recent work \cite{Lai17}, the author studied the multilinear boundedness of the higher order Calder\'on commutator. The purpose of this paper is to further generalize those results to the weighted space for its maximal type operator. Before stating our main results, let us give some notation and background.
Define the truncated higher ($n$-th) order Calder\'on commutator by
\Bes
\mathcal{C}_\eps[\nabla A_1,\cdots,\nabla A_n,f](x)=\int_{|x-y|\geq\eps} K(x-y)\Big(\prod_{i=1}^n\fr{A_{i}(x)-A_i(y)}{|x-y|}\Big)\cdot f(y)dy,
\Ees
where $n$ is a positive integer and $K$ is the Calder\'on-Zygmund convolution kernel on $\mathbb{R}^d\setminus\{0\}\, (d\ge2)$ which means that $K$ satisfies the following three conditions:
\begin{equation}\label{e:12kb}
|K(x)|\lc|x|^{-d},
\end{equation}
\begin{equation}\label{e:12K_2}
\int_{r<|x|<R}K(x)(x/|x|)^{\alp}dx=0, \text{ $\forall 0<r<R<\infty$ and $\forall \alp\in\Z_+^d $ with $|\alp|=n$},
\end{equation}
\begin{equation}\label{e:12kr}
|K(x-y)-K(x)|\lc{|y|^\del}/{|x|^{d+\del}}\ \ \text{for some $0<\del\leq1$ if}\ \ |x|>2|y|.
\end{equation}
Then we define the higher order Calder\'on commutator and its maximal operator by
\Be\label{e:12com}
\begin{split}
\mathcal{C}[\nabla A_1,\cdots,\nabla A_n,f](x)&=\lim_{\eps\rta0}\mathcal{C}_\eps[\nabla A_1,\cdots,\nabla A_n,f](x),\\
\mathcal{C}_*[\nabla A_1,\cdots,\nabla A_n,f](x)&=\sup_{\eps>0}\big|\mathcal{C}_\eps[\nabla A_1,\cdots,\nabla A_n,f](x)\big|.
\end{split}
\Ee

It is the standard context to check that these functions $\mathcal{C}[\nabla A_1,\cdots,\nabla A_n,f](x)$ and $\mathcal{C}_*[\nabla A_1,\cdots,\nabla A_n,f](x)$ are well defined for $A_1$, $\cdots$, $A_n$, $f\in C_c^\infty(\mathbb{R}^d)$ (see \eg \cite{Gra249}).
This kind of commutator was first introduced by A. P. Calder\'on \cite{Cal65} when $n=1$ and $K(x)$ is a homogeneous kernel
and later \cite{Cal77} \cite{Cal78} for the higher order one (see also \cite{CM75}, \cite{CM78}).
One can easily see that the first order Calder\'on commutator  $\mathcal{C}[\nabla A,f](x)$ is a generalization of
\Bes
\begin{split}
[A, S]f(x)=A(x)S(f)(x)-S(Af)(x)=
-\pv\fr{1}{\pi}\int_{\R}\fr{1}{x-y}\fr{A(x)-A(y)}{x-y}f(y)dy
\end{split}
\Ees
where $S=\fr{d}{dx}\circ H$ and $H$ denotes the Hilbert transform.
It is well known that the commutator $[A, S]$ and it generalization are  elementary operators in harmonic analysis, which play an important role in the theory of the Cauchy integral along Lipschitz curve in $\mathbb{C}$, the boundary value problem of elliptic equation on non-smooth domain, the Kato square root problem on $\R$ and the mixing flow problem (see \eg \cite{Cal65}, \cite{Cal78}, \cite{Fef74}, \cite{MC97}, \cite{DL18}, \cite{Gra250}, \cite{CJ87}, \cite{SSS15}, \cite{HSSS17}, \cite{Leg18} for the details).

Many classical known results about the higher order Calder\'on commutator take place in the setting of the dimension $d=1$. For example, the endpoint estimate that the $n$-th order Calder\'on commutator $\C$ maps $L^{1}(\R)\times\cdots\times L^{1}(\R)\times L^1(\R)$ to $L^{\fr{1}{1+n},\infty}(\R)$ was proved by C. P. Calder\'on \cite{CCal75} when $n=1$, Coifman and Meyer \cite{CM75} when $n=1,2$ and Duong, Grafakos and Yan \cite{DGY10} when $n\geq 1$. Here we point out that one important fact used by Coifman and Meyer \cite{CM75},  Duong, Grafakos and Yan \cite{DGY10} is that the one dimensional higher order Calder\'on commutator can be reduced to the multilinear Calder\'on Zygmund operator (see the very nice exposition \cite[Chapter 7]{Gra250} and the reference therein).
However when the dimension $d\geq 2$, things become complicated since Calder\'on commutator is a non standard multilinear Calder\'on-Zygmund operator.
If we consider the Calder\'on-Zygmund kernel $K(x)=|x|^{-d}$, then the sharp bilinear estimates (except some endpoint estimates) of the first order Calder\'on commutator in this case has been established by Fong \cite{Fon16} via the time-frequency analysis method.
For the more general Calder\'on-Zygmund kernel or even rough homogeneous kernel, the author \cite{Lai17} established all multilinear boundedness of the higher order Calder\'on commutator for the higher dimensions, especially the endpoint estimate that the $n$-th order Calder\'on commutator $\C$ maps the product of Lorentz space  $L^{d,1}(\R^d)\times\cdots\times L^{d,1}(\R^d)\times L^1(\R^d)$ to $L^{\fr{d}{d+n},\infty}(\R^d)$.

The weighted results related to the Calder\'on commutator is also only known for the case $d=1$. Duong, Gong, Grafakos, Li and Yan \cite[Theorem 4.3]{DGGLY09} proved that $\C_*$ maps $L^{q_1}(\R,w)\times\cdots\times L^{q_n}(\R,w)\times L^p(\R^d,w)$ to $L^r(\R,w)$ if $\fr{1}{r}=\big(\sum_{i=1}^n\fr{1}{q_i}\big)+\fr{1}{p}$ with $\fr{1}{n+1}<r<\infty$, $1<q_1,\cdots,q_n\leq\infty$, $1<p<\infty$ and $w\in \cap_{i=1}^nA_{q_i}(\R)\cap A_p(\R)$. For the endpoint estimate, Grafakos, Liu and Yang \cite[Corollary 1.7]{GLY11} showed that $\C_*$ maps $L^{1}(\R,w)\times\cdots\times L^{1}(\R,w)\times L^1(\R,w)$ to $L^{\fr{1}{1+n},\infty}(\R,w)$ under the assumption $w\in A_1(\R)$. The method used in Duong et al. \cite {DGGLY09} and Grafakos et al. \cite{GLY11} is both that by establishing the weighted theory for a class of multilinear Calder\'on-Zygmund operators with non-smooth kernel and then applying it to the Calder\'on commutator for the dimension $d=1$.
For the higher dimensional case of the Calder\'on commutator, no proper weighted multilinear Calder\'on-Zygmund theory can be applied directly.

In this paper, we are interested in the following weighted strong type multilinear estimate (or weighted weak type estimate) for the maximal operator of the higher order Calder\'on commutator
\Be\label{e:12mmulti}
\|\mathcal{C}_*[\nabla A_1,\cdots,\nabla A_n,f]\|_{L^r(\R^d,w)}\lc \Big(\prod_{i=1}^n\|\nabla A_i\|_{L^{q_i}(\R^d,w)}\Big)\|f\|_{L^p(\R^d,w)}
\Ee
where $\fr{1}{r}=\big(\sum_{i=1}^n\fr{1}{q_i}\big)+\fr{1}{p}$ with $1\leq q_i\leq\infty$, $(i=1,\cdots,n)$, and $1\leq p\leq\infty$. However, it is unknown whether those kind of estimates hold for the maximal Calder\'on commutator $\C_*$ even in the unweighted case. In this paper, we will work directly on the weighted space and state our main results as follows.

\begin{theorem}\label{t:12}
Let $d\geq2$ and $n$ be a positive integer. Suppose $K$ satisfies $(\ref{e:12kb}), (\ref{e:12K_2})$
and $(\ref{e:12kr})$.
Assume that  $\fr{1}{r}=\big(\sum_{i=1}^n\fr{1}{q_i}\big)+\fr{1}{p}$ with $1\leq q_i\leq\infty$ $(i=1,\cdots,n)$,  and $1\leq p\leq\infty$. Suppose $w\in \big(\bigcap_{i=1}^nA_{\max\{\fr{q_i}{d},1\}}(\R^d)\big)\cap {A_p}(\R^d)$. We may have the following conclusions:

{\rm(i).} If $\fr{d}{d+n}< r<\infty$, $1<q_i\leq\infty$ $(i=1,\cdots,n)$ and $1<p\leq\infty$, then \eqref{e:12mmulti} holds.

{\rm (ii).} If $\fr{d}{d+n}\leq r<\infty$ with $q_i=1$ for some $i=1,\cdots,n$; or $p=1$; or $r=\fr{d}{d+n}$, then the  following multilinear estimate holds
\Be\label{e:16mwek}
\|\mathcal{C}_*[\nabla A_1,\cdots,\nabla A_n,f]\|_{L^{r,\infty}(\R^d,w)}\lc \Big(\prod_{i=1}^n\|\nabla A_i\|_{L^{q_i}(\R^d,w)}\Big)\|f\|_{L^p(\R^d,w)}
\Ee
and in this case, if $q_i=d$ for some $i=1,\cdots,n$, $L^{q_i}(\R^d,w)$ in the above inequality should be replaced by $L^{d,1}(\R^d,w)$, the weighted Lorentz space. Specially, we have the following endpoint estimate
\Be\label{e:12mulendpoint}
\|\mathcal{C}_*[\nabla A_1,\cdots,\nabla A_n,f]\|_{L^{\fr{d}{d+n},\infty}(\R^d,w)}\lc \Big(\prod_{i=1}^n\|\nabla A_i\|_{L^{d,1}(\R^d,w)}\Big)\|f\|_{L^1(\R^d,w)}.
\Ee
\end{theorem}

\begin{remark}
\begin{enumerate}[(i).]
\item These results in Theorem \ref{t:12} are new even in the unweighted case when the dimension $d\geq2$.
\item When $0<r<\fr{d}{d+n}$, these multilinear strong type estimates \eqref{e:12mmulti} (or weak type estimates \eqref{e:16mwek}) do not hold for the maximal Calder\'on operator $\C_*$. In fact, some counterexamples has been constructed in \cite[Theorem 1.1]{Lai17} to show that those multilinear strong type estimates (or weak type estimates) fail even for the operator $\C$ in the case $0<r<\fr{d}{d+n}$. Thus our results in Theorem \ref{t:12} are optimal in this sense.

\item The condition of the weight $w\in\big(\bigcap_{i=1}^nA_{\max\{\fr{q_i}{d},1\}}(\R^d)\big)\cap {A_p}(\R^d)$ seems to be unnatural at the first sight, since it doesn't appear previously. However, this kind of condition is just appropriate for the higher dimensional Calder\'on commutator as we will see in our later proof. In fact $w\in A_{\max\{\fr{q_i}{d},1\}}(\R^d)$ comes from $\nabla A_i\in L^{q_i}(\R^d,w)$ and $w\in A_p(\R^d)$ comes from $f\in L^p(\R^d,w)$.
When the dimension $d=1$, \eqref{e:12mmulti} turns out to be that $\C_*$ maps $L^{q_1}(\R,w)\times\cdots\times L^{q_n}(\R,w)\times L^p(\R^d,w)$ to $L^r(\R,w)$ if $\fr{1}{n+1}<r<\infty$, $1<q_1,\cdots,q_n\leq\infty$, $1<p\leq\infty$ and $w\in \cap_{i=1}^nA_{q_i}(\R)\cap A_p(\R)$, which has been proved by Duong, Gong, Grafakos, Li and Yan \cite[Theorem 4.3]{DGGLY09} except the endpoint case $q_i=\infty$ for some $i$ or $p=\infty$. Therefore even in the one dimensional case \eqref{e:12mmulti} is new at the endpoint case $q_i=\infty$ for some $i$ or $p=\infty$. To the best knowledge of the author, \eqref{e:12mmulti} is new when $d\geq 2$.

\item Notice that $L^{1,1}(\R,w)=L^1(\R,w)$. Therefore when the dimension $d=1$, \eqref{e:12mulendpoint} is just that the maximal $n$-th order Calder\'on commutator maps $L^{1}(\R,w)\times\cdots\times L^{1}(\R,w)\times L^1(\R,w)$ to $L^{\fr{1}{1+n},\infty}(\R,w)$ under the assumption $w\in A_1(\R)$, which has been proved by Grafakos, Liu and Yang \cite[Corollary 1.7]{GLY11}. To the best knowledge of the author, \eqref{e:12mulendpoint} is new when $d\geq 2$. Although we assume that $d\geq 2$ in our main results, the proof presented in this paper is also valid for $d=1$. Therefore even when $d=1$, the proof of \eqref{e:12mmulti} and \eqref{e:12mulendpoint}  here are quite different from that by Duong, Gong, Grafakos, Li and Yan \cite{DGGLY09}, Grafakos, Liu and Yang \cite{GLY11}, thus we give new proofs of \eqref{e:12mmulti} and \eqref{e:12mulendpoint} for $d=1$.
  \item Currently, there are extensively research on seeking the optimal quantitative weighted bound for singular integral. We do not purse this topic in this paper but hope to work on it in the future work.
\end{enumerate}
\end{remark}

Notice first that if $q_i=\infty$ with $i=1,\cdots,n$, i.e. $A_i$ is a Lipschitz function, then $\C[\nabla A_1,\cdots,\nabla A_n, \cdot]$ is a standard Calder\'on Zygmund operator. By the standard weighted theory of the Calder\'on-Zygmund operator, we may easily get that $\C_*$ maps $L^\infty(\R^d,w)\times\cdots\times L^\infty(\R^d,w)\times L^p(\R^d,w)$ to $L^p(\R^d,w)$ for $1<p<\infty$ and $L^\infty(\R^d,w)\times\cdots\times L^\infty(\R^d,w)\times L^1(\R^d,w)$ to $L^{1,+\infty}(\R^d,w)$.
Recall the method used in \cite{DGGLY09} or \cite{GLY11}, by establishing the Cotlar inequality for the multilinear Calder\'on-Zygmund operator, the authors in \cite{DGGLY09} or \cite{GLY11} proved the weighted multilinear estimates for the Calder\'on-Zygmund operator and then applies them to the one-dimensional Calder\'on commutator.
There are also variants of the Cotlar inequality for the higher dimensional Calder\'on commutator, which is available only for the multilinear estimates \eqref{e:12mmulti} in the case that all $q_i>d, i=1,\cdots,n, r>1$ (see Proposition \ref{p:16maxinf}).
To deal with the remainder case, our strategy is as follows. We straightforward establish the endpoint estimates in (ii) of Theorem \ref{t:12}, which means that we need to give some weak type estimates. Note that $A_i$ belongs to the Sobolev space $W^{1,q_i}(\R^d,w)$. We will construct an {\it exceptional set\/} which satisfies the required weighted weak type estimate. And on the complementary set of {\it exceptional set\/}, the function $A_i$ is a Lipschitz function with a bound $\lam^{\fr{r}{q_i}}$. Then, roughly speaking, the strong type estimate  and the weak type $L^{1,\infty}(\R^d,w)$ boundedness (with  $q_i=\infty,i=1,\cdots,n$) of $\mathcal{C}_*[\nabla A_1,\cdots,\nabla A_n, f](x)$ could be applied on the complementary set of {\it exceptional set\/}. To construct the {\it exceptional set\/}, we will make use of the Marry Weiss maximal operator and the weighted Sobolev inequality.

This paper is organized as follows. Firstly some preliminary lemmas are presented in Section \ref{s:162}.
In Section \ref{s:163}, we give the proof of Theorem \ref{t:12}. The proof is divided into several case.  In Subsection \ref{s:1222}, we prove some strong type estimates of (i) in Theorem \ref{t:12}. The proofs of (ii) in Theorem \ref{t:12} are given in Subsections \ref{s:1223} and \ref{s:1225}. In Subsection \ref{s:1226}, we shall use the linear Marcinkiewicz interpolation with some strong type estimates of (i) and full weak type estimates of (ii)  to show the rest of (i) in Theorem \ref{t:12}.
\vskip0.24cm
\textbf{Notation}. Throughout this paper, we only consider the dimension $d\ge2$ and the letter $C$ stands for a positive finite constant which is independent of the essential variables, not necessarily the same one in each occurrence. $A\lc B$ means $A\leq CB$ for some constant $C$. By the notation $C_\eps$ means that the constant depends on the parameter $\eps$. $A\approx B$ means that $A\lc B$ and $B\lc A$.
$n$ represents the order of Calder\'on commutator. The indexes $r$, $q_1,\cdots, q_n$ and $p$ satisfy $\fr{1}{r}=\big(\sum_{i=1}^n\fr{1}{q_i}\big)+\fr{1}{p}$ with $1\leq q_i\leq\infty$ $(i=1,\cdots,n)$ and $1\leq p\leq\infty$ in the whole paper. For a set $E\subset\R^d$, we denote by $w(E)=\int_{E}w(x)dx$.  $\nabla A$ will stand for the vector $(\pari_1A,\cdots,\pari_dA)$ where $\pari_i A(x)=\pari A(x)/\pari x_i$. Define $\N_i^j=\{i,i+1,\cdots,j\}$.
Set $$\|\nabla A\|_{X}=\Big\|\Big(\sum_{i=1}^d|\pari_iA|^2\Big)^{\fr{1}{2}}\Big\|_{X}$$
for $X=L^p(\R^d,w)$ or $X=L^{d,1}(\R^d,w)$. $\Z_+$ denotes the set of all nonnegative integers and $\Z_+^d=\underbrace{\Z_+\times \cdots\times \Z_+}_d.$ For $\alp\in\Z_+^d$ and $x\in\R^d$, we define $x^\alp=x_1^{\alp_1}x_2^{\alp_2}\cdots x_d^{\alp_d}$.
\vskip1cm

\section{Some Preliminary Lemmas}\label{s:162}\quad

In this section, we will introduce the weighted properties of some operators which are useful in the proof of Theorem \ref{t:12}. Those operators include the Hardy-Littlewood maximal operator with order $\del$, the maximal sharp function operator, the Marry Weiss maximal operator and many others. And also a weighted Sobolev inequality is needed.

\begin{definition}[$A_p(\R^d)$ weight]\label{d:16a1}
A nonnegative locally integrable function $w$ on $\R^{d}$ is called to be an $A_p(\R^{d})$ weight if there exists a constant $C>0$ such that
\Be\label{ap:n-1}
\sup_{Q }\Big(\fr{1}{|Q|}\int_{Q}w(x)dx\Big)\Big(\fr{1}{|Q|}\int_Bw(x)^{-\fr{1}{p-1}}dx\Big)^{p-1}\le C<\infty,
\Ee
where the supremum is taken all cube $Q$ in $\R^{d}$. The smallest constant $C$ for \eqref{ap:n-1} holds is called the $A_p$ bound of $w$ and is denoted by $[w]_{A_p}$.
We call $w$  an $A_1(\R^d)$ weight if there exists a constant $C$ independent of $Q$ such that
\Be\label{e:16a1}
\fr{1}{|Q|}\int_{Q}w(z)dz\leq C w(y),\ \  a.e. \ \ y\in Q.
\Ee
And we set the smallest constant $C$ in \eqref{e:16a1} as $[w]_{A_1}$, which is called  the $A_1$ bound of $w$.  We also set $A_\infty(\R^d)=\bigcup_{1\leq q<\infty}A_q(\R^d)$.
\end{definition}

It is easy to see that an equivalent definition of $A_1(\R^d)$ weight is that $M(w)\leq Cw(x)$, where $M$ is the Hardy-Littlewood maximal operator. Recall the following basic fact about $A_p(\R^d)$ weight (see \cite{Gra249}):
\Bes
A_p(\R^d)\subsetneq A_q(\R^d), \ \text{if}\ 1\leq p< q\leq\infty.
\Ees

\begin{lemma}[see \cite{Jou83} or \cite{LOPTT09}]\label{l:16fs}
Suppose that $w\in A_\infty(\R^d)$. Let $0<\del,q<\infty$.  Then there exists a constant $C$ depends only on $w,\del,q$ such that
\Bes
\int_{\R^d}[M_\del f(x)]^qw(x)dx\leq C\int_{\R^d}[M_\del^{\sharp}f(x)]^qw(x)dx
\Ees
holds for any function provided that the left side integral is finite. Here $M_\del$, $M_\del^{\sharp}$ are the Hardy-Littlewood maximal operator with order $\del$ and the maximal sharp function operator, which are defined as
\Bes
\begin{split}
M_\del(f)(x)&=\sup\limits_{r>0}\Big(\frac 1{|Q(x,r)|}\int_{Q(x,r)}|f(y)|^\del dy\Big)^{\fr{1}{\del}},\\ M_\del^{\sharp}(f)(x)&=\sup_{Q\ni x}\inf_{c}\Big(\fr{1}{|Q|}\int_{Q}|f(z)-c|^\del dz\Big)^{\fr{1}{\del}},
\end{split}
\Ees
where $Q$ is a cube in $\R^d$ and $Q(x,r)$ is a cube with center $x$ and sidelength $r$.
\end{lemma}
Next we state some properties of a special maximal function introduced firstly by Mary Weiss (see \cite{CCal75}), which is defined as
$$\M(\nabla A)(x)=\sup_{h\in\R^d\setminus\{0\}}\fr{|A(x+h)-A(x)|}{|h|}.$$
\begin{lemma}\label{l:mw}
Suppose that $w\in A_{p/d}(\R^d)$ with $p>d$. Let $\nabla A\in L^p(\R^d,w)$. Then $\M$ is bounded on $L^p(\R^d,w)$, that is
$$\|\M(\nabla A)\|_{L^p(\R^d,w)}\lc\|\nabla A\|_{L^p(\R^d,w)}.$$
\end{lemma}
\begin{proof}
By using the dense argument, it is sufficient to consider $A$ as a $C^\infty$ function with compact support. Then by the result of \cite[Lemma 1.4]{CCal75}, we get that for any $q>d$,
$$\fr{|A(x)-A(y)|}{|x-y|}\lc\Big(\fr{1}{|x-y|^d}\int_{|x-z|\leq 2|x-y|}|\nabla A(z)|^{q}dz\Big)^{\fr{1}{q}}.$$

Since $w\in A_{p/d}(\R^d)$, by the revers H\"older inequality of $A_{p/d}(\R^d)$ weight (see \cite{Gra249}) and its definition, there exist $\eps>0$ such that $w\in A_{p/d-\eps}(\R^d)$ and $p/d-\eps\geq1$. Therefore we may choose $q$ in the above inequality such that $p/d-\eps=p/q$ and $d<q<p$. Applying the fact that the Hardy-Littlewood maximal operator $M$ maps $L^s(\R^d,w)$ to itself if $1<s\leq\infty$ and $w\in A_s(\R^d)$, we may get
$$\|\M(\nabla A)\|_{L^p(\R^d,w)}\lc\|M_q(\nabla A)\|_{L^p(\R^d,w)}=\|M(|\nabla A|^q)\|^{\fr{1}{q}}_{L^{p/q}(\R^d,w)}\lc\|\nabla A\|_{L^p(\R^d,w)},$$
which completes the proof.
\end{proof}
\begin{lemma}\label{l:11md}
Let $A$ be a function such that $\nabla A\in L^{d,1}(\R^d,w)$, the Lorentz space with weight $w\in A_1(\R^d)$. Then for any $\lam>0$,
$$w(\{x\in\R^d:\M(\nabla A)(x)>\lam\})\lc\lam^{-d}\|\nabla A\|^d_{L^{d,1}(\R^d,w)}.$$
\end{lemma}
\begin{proof}
By the dense argument,
it is sufficient to consider $A$ as a smooth function with compact support. Using the formula \cite[page 125, (17)]{Ste70}, one may write
\Bes
A(x)=C_d\sum_{j=1}^d\int_{\R^d}\fr{x_j-y_j}{|x-y|^d}\pari_j A(y)dy=\mathbb{K}*f(x)
\Ees
where $\mathbb{K}(x)=1/|x|^{d-1}$, $f=C_d\sum_{j=1}^dR_j(\pari_j A)$ with $R_j$s the Riesz transforms.
Notice that $w\in A_1(\R^d)\subsetneq A_p(\R^d)$ for all $1<p\leq\infty$. By applying the general form of the Marcinkiewicz interpolation theorem (see \cite[page 197, Theorem 3.15]{Ste71}), we obtain that  the Riesz transform $R_j$ maps $L^{d,1}(\R^d,w)$ to itself. Then it is easy to see that
$$\|f\|_{L^{d,1}(\R^d,w)}\lc\|\nabla A\|_{L^{d,1}(\R^d,w)}.$$
Hence to finish the proof, it is sufficient to prove that
\Be\label{e:11weissmaximal}
w(\{x\in\R^d:\M(\nabla A)(x)>\lam\})\lc\lam^{-d}\|f\|^d_{L^{d,1}(\R^d,w)}
\Ee
with $A=\mathbb{K}*f$. Below we shall show that for any $x\in\R^d$, the following estimate
$$|A(x+h)-A(x)|\lc |h|T(f)(x)$$
holds uniformly for $h\in\R^d\setminus\{0\}$ where $T$ is an operator maps $L^{d,1}(\R^d,w)$ to $L^{d,\infty}(\R^d,w)$. Once we show this, we get \eqref{e:11weissmaximal} and hence finish the proof of Lemma \ref{l:11md}. We write
\Bes
\begin{split}
A(&x+h)-A(x)\\
&=\int_{|x-y|\leq2|h|}|x+h-y|^{-d+1}f(y)dy-\int_{|x-y|\leq2|h|}|x-y|^{-d+1}f(y)dy\\
&\ \ \ \ +\int_{|x-y|>2|h|}\Big(|x+h-y|^{-d+1}-|x-y|^{-d+1}\Big)f(y)dy\\
&=I+II+III.
\end{split}
\Ees

Consider $I$ firstly. Observe that $\mathbb{K}\in L^{d',\infty}(\R^d)$ where $d'=d/(d-1)$. Set $B(x,r)=\{y\in\R^d:|x-y|\leq r\}$. Applying the rearrangement inequality (see \cite[page 74, Exercise 1.4.1]{Gra249}), we obtain that
\Bes
\begin{split}
|I|&\leq\int_{\R^d}\mathbb{K}(x+h-y)|f\chi_{B(x,2|h|)}(y)|dy\leq\int_0^\infty \mathbb{K}^*(s)(f\chi_{B(x,2|h|)})^*(s)ds\\
&\leq\Big(\int_0^\infty (f\chi_{B(x,2|h|)})^*(s)s^{\fr{1}{d}}\fr{ds}{s}\Big)\cdot\sup_{s>0}\Big(\mathbb{K}^*(s)s^{\fr{1}{d'}}\Big)\\
&\lc\|f\chi_{B(x,2|h|)}\|_{L^{d,1}(\R^d)}\|\mathbb{K}\|_{L^{d',\infty}(\R^d)},
\end{split}
\Ees
where $f^*$ stands for the decreasing rearrangement of $f$.
Applying the definition of Lorentz space, we may get that $\|\chi_E\|_{L^{d,1}(\R^d)}=\|\chi_E\|_{L^d(\R^d)}$ holds for any characteristic function $\chi_E$ of set $E$ of finite measure, thus $\|\chi_{B(x,2|h|)}\|_{L^{d,1}(\R^d)}=C_d|h|$. Then we obtain that
$$|I|\lc |h|\Lam(f)(x),\ \ \text{where}\ \ \Lam(f)(x)=\sup_{r>0}\fr{\|f\chi_{B(x,r)}\|_{L^{d,1}(\R^d)}}{\|\chi_{B(x,r)}\|_{L^{d,1}(\R^d)}}.$$
In the following it is sufficient to show that the operator $\Lam$ maps $L^{d,1}(\R^d,w)$ to $L^{d,\infty}(\R^d,w)$. Note that $L^{d,1}(\R^d,w)$ is a Banach space (see \eg \cite[page 204, Theorem 3.22]{Ste71}), it suffices to show that $\Lam$ is restricted of type $(d,d)$, thus is $\|\Lam(\chi_E)\|_{L^{d,\infty}(\R^d,w)}\lc w(E)^{\fr{1}{d}}$ (see \eg \cite[page 62, Lemma 1.4.20]{Gra249}). However
in this case, the proof is equivalent to show that
$$w(\{x\in\R^d: M(\chi_E)(x)>\lam\})\lc\lam^{-1}\|\chi_E\|_{L^1(\R^d,w)},$$
where $M$ is the Hardy-Littlewood maximal operator. Since $M$ is weighted weak type (1,1) if $w\in A_1(\R^d)$, hence we prove that $\Lam$ maps $L^{d,1}(\R^d,w)$ to $L^{d,\infty}(\R^d,w)$.

Next consider $II$. Observe that the kernel $k(x):=\eps^{-1}|x|^{-d+1}\chi_{\{|x|\leq \eps\}}$ is radial non-increasing and $L^1$ integrable in $\R^d$, we get
$$|II|\lc \|k\|_{L^1(\R^d)}|h|M(f)(x).$$
Notice that $L^{p,1}(\R^d,w)\subset L^p(\R^d,w)$ and $M$ maps $L^p(\R^d,w)$ to itself for $1<p<\infty$. Hence we get that $M$ maps $L^{d,1}(\R^d,w)$ to $L^{d,\infty}(\R^d,w)$.

Finally consider $III$. Notice that it suffices to consider $|x-y|>2|h|$. Then applying the Taylor expansion of $|x-y+h|^{-d+1}$, we get
\Be\label{e:11taylor}
\fr{1}{|x-y+h|^{d-1}}-\fr{1}{|x-y|^{d-1}}=(-d+1)\sum_{j=1}^dh_j\fr{x_j-y_j}{|x-y|^{d+1}}+R(x,y,h)
\Ee
where  the remainder term $R(x,y,h)$ in the Taylor expansion satisfies
$$|R(x,y,h)|\leq C|h|^2|x-y|^{-d-1}\ \text{if } |x-y|>2|h|.$$
Inserting \eqref{e:11taylor} into the term $III$ with the above estimate of $R(x,y,h)$, we conclude that
$$|III|\lc |h|\sum_{j=1}^d R_j^*(f)(x)+|h|^2\int_{|x-y|>2|h|}|x-y|^{-d-1}|f(y)|dy$$
where $R_j^*$ is the maximal Riesz transform defined by
$$R_j^*(f)(x)=\sup_{\eps>0}\Big|\int_{|x-y|>\eps}\fr{x_j-y_j}{|x-y|^{d+1}}f(y)dy\Big|.$$
Since $R_j^*$ is bounded on $L^p(\R^d,w)$ for $1<p<\infty$, we immediately obtain that $R_j^*$ maps $L^{d,1}(\R^d,w)$ to $L^{d,\infty}(\R^d,w)$.
The second term which controls $III$ can be dealt similar to that of the estimate of $II$ once we observe that $\eps|x|^{-d-1}\chi_{\{|x|>\eps\}}$ is radial non-increasing and $L^1$ integrable.
\end{proof}

In the following, we introduce a weighted Sobolev inequality and a key weighted weak type estimate for $\nabla A\in L^p(\R^d,w)$ with $1\leq p<d$.
Define the weighted Hardy-Littlewood maximal operator of order $p$ $M_{w,p}$ and  the weighted maximal operator $\mathfrak{M}_{w,s}$ by
\Bes
\begin{split}
M_{w,p}(f)(x)&=\sup\limits_{r>0}\bigg(\frac 1{w(Q(x,r))}\int_{Q(x,r)}|f(y)|^pw(y)dy\bigg)^{1/p},\\
\mathfrak{M}_{w,s}(\nabla A)(x)&=\sup_{r>0}\Big(\fr{1}{w(Q(x,r))}\int_{Q(x,r)}\Big|\fr{A(x)-A(y)}{r}\Big|^{s}w(y)dy\Big)^{1/s},\\
\end{split}
\Ees
where $Q(x,r)$ is a cube with center $x$ and sidelength $r$.

\begin{lemma}[see \cite{DS90}]\label{l:16sobolev}
If $w$ is an $A_1(\R^d)$ weight, then the following weighted Sobolev inequality $$\Big(\int_{\R^d}|g(x)|^sw(y)dy\Big)^{\fr{1}{s}}\leq C\Big(\int_{\R^d}\big[w(x)^{-\fr{1}{d}}|\nabla g(x)|\big]^pw(x)dx\Big)^{\fr{1}{p}}$$
holds for $1\leq p<d$ and $1/s=1/p-1/d$, where $g$ is a $C^1$ function with compact support and the constant $C$ does not depend on $g$.
\end{lemma}

\begin{lemma}\label{l:12qleqd}
Let $w\in A_1(\R^d)$ and $\nabla A\in L^p(\R^d,w)$ with $1\leq p<d$. Set $1/s=1/p-1/d$. Then we have
\Bes
w(\{x\in\R^d:\mathfrak{M}_{w,s}(\nabla A)(x)>\lam\})\lc\lam^{-p}\|\nabla A\|_{L^p(\R^d,w)}^p.
\Ees
\end{lemma}
\begin{proof}
By using a standard limiting argument, we only need to consider $A$ as a $C^\infty$ function with compact support.
Fix a cube $Q(x,r)$. Choose a $C_c^\infty$ function $\phi$ such that $\phi(y)\equiv1$ if $y\in Q(x,r)$, supp$\phi\subset Q(x,2r)$ and $\|\nabla \phi\|_{L^\infty(\R^d)}\lc r^{-1}$. Consider the auxiliary function $\phi(y)(A(x)-A(y))$ where $x$ is fixed and $y$ is the variable. 
Using the weighted Sobolev inequality in Lemma \ref{l:16sobolev} and the property \eqref{e:16a1} of $A_1(\R^d)$ weight, one may get that
\Bes
\begin{split}
\Big(\int_{Q(x,r)}|A(x)-A(y)|^sw(y)dy\Big)^{\fr{1}{s}}
&\lc\Big[\int_{\R^d}\big[\nabla_y(\phi(y)(A(x)-A(y)))\big]^{{p}}w(y)^{1-\fr{p}{d}}dy\Big]^{\fr{1}{p}}\\
&\lc\Big(\int_{Q(x,2r)}\big|\nabla A(y)\big|^{{p}}w(y)^{1-\fr{p}{d}}dy\Big)^{\fr{1}{p}}\\
&\quad+\Big(\int_{Q(x,2r)\setminus Q(x,r)}\Big|\fr{A(x)-A(y)}{r}\Big|^{{p}}w(y)^{1-\fr{p}{d}}dy\Big)^{\fr{1}{p}}\\
&\lc\Big[\fr{w(Q(x,2r))}{|Q(x,2r)|}\Big]^{-\fr{1}{d}}\Big[\Big(\int_{Q(x,2r)}\big|\nabla A(y)\big|^{{p}}w(y)dy\Big)^{\fr{1}{p}}\\
&\quad+\Big(\int_{Q(x,2r)\setminus Q(x,r)}\Big|\fr{A(x)-A(y)}{r}\Big|^{{p}}w(y)dy\Big)^{\fr{1}{p}}\Big].
\end{split}
\Ees
The above estimate, via the doubling property of $w(x)dx$ (i.e. $w(2Q)\lc w(Q)$, see \cite{Gra249}) and $\fr{1}{s}=\fr{1}{p}-\fr{1}{d}$, yields that
$$\Big(\fr{1}{w(Q(x,r))}\int_{Q(x,r)}\Big|\fr{A(x)-A(y)}{r}\Big|^{s}w(y)dy\Big)^{\fr{1}{s}}\lc M_{w,p}(\nabla A)(x)+S_{p}(\nabla A)(x),$$
where
$$S_{p}(\nabla A)(x):=\Big[\fr{1}{w(Q(x,2r))}\int_{Q(x,2r)\setminus Q(x,r)}\Big|\fr{A(x)-A(y)}{r}\Big|^{{p}}w(y)dy\Big]^{\fr{1}{p}}.$$
Again using the fact that $w(x)dx$ satisfies the doubling property, one may see that the Hardy-Littlewood maximal operator $M_{w,1}$ with the weight $w$ is of weak type (1,1), thus is $M_{w,1}$ maps $L^1(\R^d,w)$ to $L^{1,\infty}(\R^d,w)$. Then we get that $M_{w,p}$ maps $L^p(\R^d,w)$ to $L^{p,\infty}(\R^d,w)$.  Therefore to complete the proof, it is enough to show that $S_{p}(\nabla A)(x)\lc T(\nabla A)(x)$ with $T$ mapping $L^p(\R^d,w)$ to $L^{p,\infty}(\R^d,w)$.

Below we give some explicit estimates of $A(x)-A(y)$ similar to that in the proof of Lemma \ref{l:11md}. By the formula given in \cite[page 125, (17)]{Ste70}, we may write
\Bes
A(x)=C_d\sum_{j=1}^d\int_{\R^d}\fr{x_j-y_j}{|x-y|^d}\pari_j A(y)dy.
\Ees

Split $A(x)-A(y)$ into three terms as follows,
\Be\label{e:16axy}
\begin{split}
A(&x)-A(y)\\
&=C_d\sum_{j=1}^d\Big[\int_{|x-z|\leq2|x-y|}\fr{x_j-z_j}{|x-z|^{d}}\pari_jA(z)dz-\int_{|x-z|\leq2|x-y|}\fr{y_j-z_j}{|y-z|^d}\pari_jA(z)dz\\
&\ \ \ \ +\int_{|x-z|>2|x-y|}\Big(\fr{x_j-z_j}{|x-z|^{d}}-\fr{y_j-z_j}{|y-z|^{d}}\Big)\pari_jA(z)dz\Big]\\
&=I(x)+II(x)+III(x).
\end{split}
\Ee
Plug the above three terms back into $S_p(\nabla A)(x)$ and define these three terms as $S_{p,1}(\nabla A)(x)$, $S_{p,2}(\nabla A)(x)$ and  $S_{p,3}(\nabla A)(x)$ respectively.

Let us first consider $S_{p,1}(\nabla A)(x)$. By applying the H\"older inequality,
$$|I(x)|^{{p}}\lc|x-y|^{p-1}\Big(\int_{|x-z|\leq2|x-y|}\fr{|\nabla A(z)|^p}{|x-z|^{d-1}}dz\Big).$$
Plugging the above inequality into $S_{p,1}(\nabla A)(x)$ with $|x-y|\approx r$, and then using the kernel $k(x)=\eps^{-1}|x|^{-d+1}\chi_{\{|x|\leq C\eps\}}$ is a radial non-increasing function and $L^1$ integrable in $\R^d$, we get that
\Be
\begin{split}
S_{p,1}(\nabla A)(x)&\lc \Big(r^{-1}\int_{|x-z|\lc r}\fr{|\nabla A(z)|^p}{|x-z|^{d-1}}dz\Big)^{\fr{1}{p}}\lc M_p(\nabla A)(x).
\end{split}
\Ee
It is easy to see that $M_p$ maps $L^p(\R^d,w)$ to $L^{p,\infty}(\R^d,w)$ with $w$ an $A_1(\R^d)$ weight is equivalent to that the Hardy-Littlewood maximal operator $M$ maps $L^1(\R^d,w)$ to $L^{1,\infty}(\R^d,w)$, which is however well known.

Next we consider $S_{p,2}(\nabla A)(x)$. By using the H\"older inequality to deal with $II(x)$ as those of $I(x)$, then applying $|x-y|\approx r$ and the Fubini theorem, we get
\Be
\begin{split}
S_{p,2}(\nabla A)(x)&\lc\Big[\fr{1}{w(Q(x,2r))}\int_{Q(x,2r)\setminus Q(x,r)}\fr{1}{r}\Big(\int_{|x-z|\lc r}\fr{|\nabla A(z)|^p}{|y-z|^{d-1}}dz\Big) w(y)dy\Big]^{\fr{1}{p}}\\
&\lc\Big[\fr{1}{w(Q(x,2r))}\int_{|x-z|\lc r}r^{-1}\Big(\int_{|y-z|\lc r}\fr{w(y)}{|y-z|^{d-1}}dy\Big)|\nabla A(z)|^pdz\Big]^{\fr{1}{p}}\\
&\lc\Big[\fr{1}{w(Q(x,2r))}\int_{|x-z|\lc r}M(w)(z)|\nabla A(z)|^pdz\Big]^{\fr{1}{p}}\\
&\lc\Big[\fr{1}{w(Q(x,2r))}\int_{|x-z|\lc r}|\nabla A(z)|^pw(z)dz\Big]^{\fr{1}{p}}\lc M_{w,p}(\nabla A)(x),
\end{split}
\Ee
where in the third inequality we use again the fact that the kernel function $k(x)={\eps}^{-1}{|x|^{-d+1}}\chi_{\{|x|\leq C\eps\}}$ is a radial non-increasing function and $L^1$ integrable in $\R^d$, the last second inequality follows from \eqref{e:16a1}. As showed previously,  $M_{w,p}$ maps $L^p(\R^d,w)$ to $L^{p,\infty}(\R^d,w)$.

We consider $S_{p,3}(\nabla A)(x)$. Set $\mathcal{K}_j(x)=\fr{x_j}{|x|^d}$. Notice that $|x-z|>2|x-y|$. Applying the Taylor expansion of $\mathcal{K}_j(x-z)$, we may get
\Bes
\mathcal{K}_j(x-z)-\mathcal{K}_j(y-z)=\sum_{i=1}^d(x_i-y_i)\pari_i \mathcal{K}_j(x-z)+R(x,y,z)
\Ees
where the Taylor expansion's remainder term $R(x,y,z)$ satisfies
$$|R(x,y,z)|\lc |x-y|^2|x-z|^{-d-1},\ \forall\ |x-z|>2|x-y|.$$
Plunge the Taylor expansion's main term and reminder term into $S_{p,3}(\nabla A)(x)$  and split $S_{p,3}(\nabla A)(x)$  as two terms $S_{p,3,m}(\nabla A)(x)$ (related to main term) and $S_{p,3,r}(\nabla A)(x)$ (related to reminder term), respectively. Then by $|x-y|\approx r$, we have the following estimate of $S_{p,3,m}(\nabla A)(x)$,
\Bes
\begin{split}
&S_{p,3,m}(\nabla A)(x)\\&\lc\Big[\fr{1}{w(Q(x,2r))}\int_{Q(x,2r)}\Big(\sum_{j=1}^d\sum_{i=1}^d\Big|\int_{|x-z|>2|x-y|}{\pari_i\mathcal{K}_j(x-z)\pari_jA(z)dz}\Big|\Big)^pw(y)dy\Big]^{\fr{1}{p}}\\
&\lc \sum_{j=1}^d\sum_{i=1}^d T_{i,j}^*(\pari_j A)(x),
\end{split}
\Ees
where the maximal singular integral operator $T_{i,j}^*(f)(x)$ is defined as follows
\Be\label{e:16ijcz}
T_{i,j}^*(f)(x)=\sup_{\eps>0}\Big|\int_{|x-y|>\eps}\pari_i\mathcal{K}_j(x-y)f(y)dy\Big|.
\Ee
One can easily check that the kernel $\pari_i\mathcal{K}_j(x-y)$ is a standard Calder\'on-Zygmund convolution kernel which satisfies \eqref{e:12kb}, \eqref{e:12kr} and has mean value zero on $\S^{d-1}$. Then by the standard weighted Calder\'on-Zygmund theory (see \cite{Gra249}), $T_{i,j}^*$ is bounded on $L^p(\R^d,w)$. So $T_{i,j}^*$ maps $L^p(\R^d,w)$ to $L^{p,\infty}(\R^d,w)$.

Finally one may apply the method similar to that of $I$ to handle the reminder term $S_{p,3,r}(\nabla A)(x)$. Indeed, by the H\"older inequality and $|x-y|\approx r$, we get
\Bes
\begin{split}
S_{p,3,r}(\nabla A)(x)
&\lc\Big[\fr{1}{w(Q(x,2r))}\int_{Q(x,2r)\setminus Q(x,r)}r\Big(\int_{r\lc|x-z|}\fr{|\nabla A(z)|^p}{|x-z|^{d+1}}dz\Big)w(y)dy\Big]^{\fr{1}{p}}\\
&\lc M_p(\nabla A)(x),
\end{split}
\Ees
where in the last inequality we use that the function $\eps|x|^{-d-1}\chi_{\{|x|>\eps\}}$ is radial non-increasing and $L^1$ integrable. As showed in the estimate of $I$, we get that $M_p$ maps $L^p(\R^d,w)$ to $L^{p,\infty}(\R^d,w)$. Hence we complete the proof.
\end{proof}
\begin{remark}
When giving an estimate in \eqref{e:16axy}, we in fact prove that the following inequality
\Be\label{e:16adiff}
\fr{|A(x)-A(y)|}{|x-y|}\lc M(\nabla A)(x)+M(\nabla A)(y)+\sum_{i=1}^d\sum_{j=1}^dT_{i,j}^*(\pari_j A)(x)
\Ee
holds for almost every $x,y\in\R^d$ if $A$ is a $C_c^\infty$ function, where $T_{i,j}^*$ is defined in \eqref{e:16ijcz}.
\end{remark}
\begin{lemma}\label{l:12disq1infty}
Let $\{Q_k\}_{k}$ be the disjoint cubes in $\R^d$. Denote by $l(Q_k)$ the side length of $Q_k$. Define the operator $T_s$ as
$$T_{s}(f)(x)=\sum_k\int_{Q_k}\fr{l(Q_k)^{s}}{[l(Q_k)+|x-y|]^{d+s}}|f(y)|dy.$$
Suppose that $1\leq q\leq\infty$, $w\in A_q(\R^d)$ and $f\in L^q(\R^d,w)$. Then for any $s>0$, we get that
\Bes
\|T_{s}(f)\|_{L^q(\R^d,w)}\lc\|f\|_{L^q(\R^d,w)}.
\Ees
\end{lemma}
\begin{proof}
If $q=1$, Lemma \ref{l:12disq1infty} just follows from the property \eqref{e:16a1} of $A_1(\R^d)$ weight and the Fubini theorem. In fact, we have
\Bes
\begin{split}
\|T_s(f)\|_{L^1(\R^d,w)}&\leq\sum_{Q_k}\int_{Q_k}\Big[\int_{\R^d}\fr{w(x)\cdot l(Q_k)^{s}}{[l(Q_k)+|x-y|]^{d+s}}dx\Big]\cdot|f(y)|dy\\
&\lc\sum_{Q_k}\int_{Q_k}M(w)(y)\cdot|f(y)|dy\\
&\lc[w]_{A_1}\sum_{Q_k}\int_{Q_k}|f(y)|w(y)dy\lc[w]_{A_1}\|f\|_{L^1(\R^d,w)},
\end{split}
\Ees
where the second inequality follows from that splitting the kernel $\fr{l(Q_k)^{s}}{[l(Q_k)+|x-y|]^{d+s}}$ into two parts according whether $|x-y|\leq l(Q_k)$ or $|x-y|> l(Q_k)$, the third inequality follows from the property \eqref{e:16a1} and in the last inequality we use that $Q_k$s are cubes disjoint each other. After we establish $T_s$ is bounded on $L^1(\R^d,w)$ with bound $[w]_{A_1}$,  the proof of the case $1<q<\infty$ just follows from the famous extrapolation theorem (see \eg Theorem 7.5.3 in \cite{Gra249}).
If $q=\infty$, apply the Fubini theorem,
\Bes
|T_s(f)(x)|\leq\sum_{Q_k}\|f\|_{L^\infty(Q_k)}\sup_{x\in\R^d}\int_{Q_k}\fr{l(Q_k)^{s}}{[l(Q_k)+|x-y|]^{d+s}}dy\lc\|f\|_{L^\infty(\R^d)}.
\Ees
Then $T_s$ is bounded on $L^\infty(\R^d,w)$ is just a consequence of the chain of inequalities:
\Be\label{e:16minf}
\|T_s(f)\|_{L^\infty(\R^d,w)}\lc\|T_s(f)\|_{L^\infty(\R^d)}\lc\|f\|_{L^\infty(\R^d)}\lc\|f\|_{L^\infty(\R^d,w)},
\Ee
which can be proved as follows. Notice that we have the equivalent definition of $L^\infty(\R^d,\mu)$: $\|f\|_{L^\infty(\R^d,\mu)}=\sup\{\alp: \mu(\{x\in\R^d:|f(x)|>\alp\})>0\}.$
The first inequality in \eqref{e:16minf} follows from the fact that $w(E)>0$ implies $|E|>0$. Likewise, the last inequality in \eqref{e:16minf} follows from the fact that $|E|>0$ implies $w(E)>0$, because $w(x)=0$ only for the points in a set of measure zero by the definition of $A_\infty(\R^d)$ weight. Hence we complete the proof.
\end{proof}
\begin{remark}
By the last argument above, for any $w\in A_\infty(\R^d)$, the follow equality
$$\|f\|_{L^\infty(\R^d)}\approx \|f\|_{L^\infty(\R^d,w)}$$
holds. We will straightforward apply this equivalence many times later.
\end{remark}

\section {Proof of Theorem \ref{t:12}}\label{s:163}
\vskip0.24cm

\subsection{Some basic strong type multilinear estimates}\label{s:1222}\quad
\vskip0.24cm
In the following, we begin to give the proof of Theorem \ref{t:12}. In this subsection, we will first show our theorem in the case $q_1=\cdots=q_n=\infty$, $r=p\in [1,\infty)$ which is not quite complicated and the case $d<q_1,\cdots,q_n\leq\infty, 1<r<\infty, p=\infty$.
\begin{prop}\label{p:12strongr}
Let $q_i=\infty$ with $i=1,\cdots,n$, $1\leq r=p<\infty$.  We have the following conclusions:
\begin{enumerate}[(i).]
\item If $p=r\in (1,\infty)$, $w\in A_p(\R^d)$, then
\Bes
\|\C_*[\nabla A_1,\cdots,\nabla A_n, f]\|_{L^p(\R^d,w)}\lc\Big(\prod_{i=1}^n\|\nabla A_i\|_{L^{\infty}(\R^d,w)}\Big)\|f\|_{L^p(\R^d,w)}.
\Ees
\item If $p=r=1$, $w\in A_1(\R^d)$, then
\Bes
\|\C_*[\nabla A_1,\cdots,\nabla A_n, f]\|_{L^{1,\infty}(\R^d,w)}\lc\Big(\prod_{i=1}^n\|\nabla A_i\|_{L^{\infty}(\R^d,w)}\Big)\|f\|_{L^1(\R^d,w)}.
\Ees
\end{enumerate}
\end{prop}
\begin{proof}
The proof of this lemma is quite standard, so we just give some key steps. When $q_1=\cdots=q_n=\infty$, $A_i$ is a Lipschitz function for $i=1,\cdots,n$. Fix all $A_i$. Observe that the kernel
\Be\label{e:16kernelinf}
\mathfrak{K}(x,y):=K(x-y)\Big(\prod_{i=1}^n\fr{A_i(x)-A_i(y)}{|x-y|}\Big)
\Ee
is a standard Calder\'on-Zygmund kernel satisfying the boundedness condition and regularity condition with bound $\prod_{i=1}^n\|\nabla A_i\|_{L^\infty(\R^d)}$ (see \eg \cite[Definition 4.1.2]{Gra250}). Then we may have the following $L^2$ boundedness
\Be\label{e:16pvcom}
\|\C[\nabla A_1,\cdots,\nabla A_n, f]\|_{L^2(\R^d)}\lc\Big(\prod_{i=1}^n\|\nabla A_i\|_{L^{\infty}(\R^d)}\Big)\|f\|_{L^2(\R^d)}.
\Ee
This in fact can be seen by using the famous $T1$ theorem (see \cite{Gra250}) or by applying the mean value formula $$\fr{A_i(x)-A_i(y)}{|x-y|}=\int_{0}^1\Big\langle{\fr{x-y}{|x-y|}},{\nabla A_i(sx+(1-s)y)}\Big\rangle ds$$
to reduce the operator $\C$ to the following operator introduced by Christ and Journ\'e \cite{CJ87}
$$\C_{CJ}[a_1,\cdots,a_n,f](x)=\pv\int_{\R^d}k(x-y)(\prod_{i=1}^nm_{x,y}a_i)f(y)dy$$ which maps $L^{\infty}(\R^d)\times\cdots\times L^{\infty}(\R^d)\times L^2(\R^d)$ to $L^2(\R^d)$ (see \cite{CJ87}). Here in the above operator $k(x-y)$ is a standard Calder\'on-Zygmund kernel and $m_{x,y}a=\int_0^1a(sx+(1-s)y)dy$.
Then the rest of the proof just follows from  the standard weighted Cader\'on-Zygmund theory (see \cite[Theorem 7.4.6.]{Gra249}).
\end{proof}

\begin{prop}\label{e:16unwrgeq1}
Suppose that $1<r<\infty$, $d<q_1,\cdots,q_n\leq\infty$ and $\fr{1}{r}=\sum_{i=1}^n\fr{1}{q_i}$. Let $w\in\bigcap_{i=1}^n A_{\fr{q_i}{d}}(\R^d)$. Assume that $\nabla A_i\in L^{q_i}(\R^d,w), i=1,\cdots,n$ and $f\in L^\infty(\R^d,w)$. Then we get
$$\|\C[\nabla A_1,\cdots,\nabla A_n, f]\|_{L^r(\R^d,w)}\lc\Big(\prod_{i=1}^{n}\|\nabla A_i\|_{L^{q_i}(\R^d,w)}\Big)\|f\|_{L^\infty(\R^d,w)}.$$
\end{prop}
\begin{proof}
By the standard limiting arguments, it is enough to consider that each $A_i$ are $C_c^\infty$ functions and $f$ is bounded compact function. Then one can easily check that $\int_{\R^d}\big[M_\del\big(\C[\nabla A_1,\cdots,\nabla A_n, f]\big)(x)\big]^rw(x)dx$ is  finite (for example one may use the method in \cite[page 1248]{LOPTT09} to show this). Therefore using the Fefferman-Stein inequality in Lemma \ref{l:16fs}, we may get that for any $\del>0$,
\Bes
\begin{split}
\|\C[\nabla A_1,\cdots,\nabla A_n, f]\|_{L^r(\R^d,w)}&\leq
\|M_\del\big(\C[\nabla A_1,\cdots,\nabla A_n, f]\big)\|_{L^r(\R^d,w)}\\
&\lc\|M_\del^{\sharp}\big(\C[\nabla A_1,\cdots,\nabla A_n, f]\big)\|_{L^r(\R^d,w)}.
\end{split}
\Ees

In the following, we need to give an estimate of the maximal sharp function. Fix $x$ and a cube $Q\ni x$. Define $f_1=f\chi_{3Q}$ and $f_2=f-f_1$. Then write
\Bes
\C[\nabla A_1,\cdots, \nabla A_n,f](z)=\C[\nabla A_1,\cdots, \nabla A_n,f_1](z)+\C[\nabla A_1,\cdots, \nabla A_n,f_2](z).
\Ees

Choose a constant $c=\C[\nabla A_1,\cdots, \nabla A_n,f_2](x)$ in the maximal sharp function. Then we see that this maximal sharp function
$M_\del^{\sharp}(\C[\nabla A_1,\cdots, \nabla A_n, f])(x)$ is bounded by the following two functions
\Bes
\begin{split}I(x)+&II(x):=\sup_{Q\ni x}\Big(\fr{1}{|Q|}\int_{Q}|\C[\nabla A_1,\cdots, \nabla A_n,f_1](z)|^\del dz\Big)^{\fr{1}{\del}}\\
&+\sup_{Q\ni x}\Big(\fr{1}{|Q|}\int_{Q}|\C[\nabla A_1,\cdots, \nabla A_n,f_2](z)-\C[\nabla A_1,\cdots, \nabla A_n,f_2](x)|^\del dz\Big)^{\fr{1}{\del}}.
\end{split}
\Ees

We first consider the above first function $I(x)$. Define $\tilde{A}_i=A_i\chi_{3Q}$. Then for any $z\in Q$, we may write
\Bes
\C[\nabla A_1,\cdots, \nabla A_n,f_1](z)=\C[\nabla \tilde{A}_1,\cdots, \nabla \tilde{A}_n,f_1](z).
\Ees
Choose $\del\leq d/n$. Applying the H\"older inequality, strong type multilinear estimate (see \cite[Theorem 1.1]{Lai17}) and the definition of $\tilde{A}_i$, we may get
\Bes
\begin{split}
\Big(\fr{1}{|Q|}\int_{Q}|\C[\nabla A_1,\cdots, \nabla A_n,f_1](z)|^\del dz\Big)^{\fr{1}{\del}}&\lc\|\C[\nabla \tilde{A}_1,\cdots, \nabla \tilde{A}_n,f_1]\|_{L^{\fr{d}{n}}(Q,\fr{dx}{|Q|})}\\
&\lc\|f_1\|_{L^\infty(\R^d)}\prod_{i=1}^n\|\nabla \tilde{A}_i\|_{L^d(\R^d,\fr{dx}{|Q|})}\\
&\lc\|f\|_{L^\infty(\R^d)}\prod_{i=1}^d M_d(\nabla A_i)(x),
\end{split}
\Ees
where $M_d$ is the Hardy-Littlewood maximal operator of order $d$.
Notice that $w\in\bigcap_{i=1}^n A_{q_i/d}(\R^d)$, by using the weighted boundedness of the Hardy-Littlewood maximal operator, one may easily get that $\|M_d(\nabla A_i)\|_{L^{q_i}(\R^d,w)}\lc\|\nabla A_i\|_{L^{q_i}(\R^d,w)}$. Therefore
$$\|I\|_{L^r(\R^d,w)}\lc\Big(\prod_{i=1}^{n}\|\nabla A_i\|_{L^{q_i}(\R^d,w)}\Big)\|f\|_{L^\infty(\R^d,w)}.$$

Next we turn to $II(x)$. Write
\Bes
\C[\nabla A_1,\cdots, \nabla A_n,f_2](z)-\C[\nabla A_1,\cdots, \nabla A_n,f_2](x)=
\int_{(3Q)^c}\big[\mathfrak{K}(z,y)-\mathfrak{K}(x,y)\big]f(y)dy
\Ees
where $\mathfrak{K}(x,y):=K(x-y)\prod_{i=1}^n\fr{A_i(x)-A_i(y)}{|x-y|}$. Then write
\Bes
\begin{split}
\mathfrak{K}&(z,y)-\mathfrak{K}(x,y)\\
&=\Big(\fr{K(z-y)}{|z-y|^n}-\fr{K(x-y)}{|x-y|^n}\Big){\prod_{i=1}^n(A_i(z)-A_i(y))}\\
&\quad+\fr{K(x-y)}{|x-y|^n}\Big(\prod_{i=1}^n(A_i(z)-A_i(y))-\prod_{i=1}^n(A_i(x)-A_i(y))\Big)\\
&=:\mathfrak{K}_1(z,x,y)+\mathfrak{K}_2(z,x,y).
\end{split}
\Ees

We consider the term $\mathfrak{K}_1(z,x,y)$. Notice that $x, z\in Q$ and $y\in (3Q)^c$, then $|z-y|\approx |x-y|$. By the regularity condition \eqref{e:12kr} and the formula \eqref{e:16adiff}, we get that
\Bes
|\mathfrak{K}_1(z,x,y)|\lc\fr{(l(Q))^\del}{|x-y|^{d+\del}}\prod_{i=1}^d[M(\nabla A_i)(z)+T(\nabla A_i)(y)],
\Ees
where here and in the following, $T$ is the sum of combination of the Hardy-Littlewood maximal operator and maximal singular integral $T_{i,j}^*$ defined in \eqref{e:16ijcz}, which both map $L^q(\R^d,w)$ to itself for $1<q<\infty$.

Next we consider the term $\mathfrak{K}_2(z,x,y)$. We may split $\mathfrak{K}_{2}(z,x,y)$ into $n$ terms and apply \eqref{e:16adiff},
\Bes
\begin{split}
\mathfrak{K}_2(z,x,y)&\lc\fr{1}{|x-y|^{d+n}}\Big|\sum_{i=1}^n[A_i(z)-A_i(x)]\prod_{k=1}^{i-1}[A_k(x)-A_i(y)]\prod_{k=i+1}^{n}[A_k(z)-A_k(y)]\Big|\\
&\lc\fr{l(Q)}{|x-y|^{d+1}}\prod_{i=1}^n[M(\nabla A_i)(z)+T(\nabla A_i)(x)+T(\nabla A_i)(y)].
\end{split}
\Ees
Combining these estimates of $\mathfrak{K}_1$ and $\mathfrak{K}_2$, we get
\Bes
\begin{split}
|\mathfrak{K}&(z,y)-\mathfrak{K}(x,y)|\\
&\lc\fr{(l(Q))^\del}{|x-y|^{d+\del}}\prod_{i=1}^n[M(\nabla A_i)(z)+T(\nabla A_i)(x)+T(\nabla A_i)(y)]\\
&\lc\fr{(l(Q))^\del}{|x-y|^{d+\del}}\sum_{\N_1^n}\prod_{i\in N_1}M(\nabla A_i)(z)\prod_{i\in N_2}T(\nabla A_i)(x)\prod_{i\in N_3}T(\nabla A_i)(y),
\end{split}
\Ees
where in the last inequality we divide $\N_{1}^n=N_1\cup N_2\cup N_3$ with $\N_{1}^n=\{1,\cdots,n\}$ and $N_1$, $N_2$, $N_3$ non intersecting each other. Plugging the above estimates into $II(x)$ and applying the H\"older inequality, we get that
\Bes
\begin{split}
II(x)&\lc\sum_{\N_1^n}\big[\prod_{i\in N_2}T(\nabla A_i)(x)\big]\fr{1}{|Q|}\int_{Q}\int_{(3Q)^c}\fr{(l(Q))^\del}{|x-y|^{d+\del}}\big[\prod_{i\in N_3}T(\nabla A_i)(y)\big]f(y)dy\\
&\quad\quad\times\big[\prod_{i\in N_1}T(\nabla A_i)(z)\big]dz\\
&\lc\|f\|_{L^\infty(\R^d)}\sum_{\N_1^n}\big[\prod_{i\in N_2}T(\nabla A_i)(x)\big]M\big[\prod_{i\in N_3}T(\nabla A_i)\big](x)\cdot M\big[\prod_{i\in N_1}T(\nabla A_i)\big](x).
\end{split}
\Ees
Now using the H\"older inequality and the fact that $M$, $T$ are bounded on $L^q(\R^d,w)$ for $1<q<\infty$, we get
$$\|II\|_{L^r(\R^d,w)}\lc \Big(\prod_{i=1}^{n}\|\nabla A_i\|_{L^{q_i}(\R^d,w)}\Big)\|f\|_{L^\infty(\R^d,w)},$$
which completes the proof.
\end{proof}

\begin{prop}\label{p:16maxinf}
Suppose that $1<r<\infty$, $d<q_1,\cdots,q_n\leq\infty$, $p=\infty$ and $\fr{1}{r}=\sum_{i=1}^n\fr{1}{q_i}$. Let $w\in\bigcap_{i=1}^n A_{q_i/d}(\R^d)$. Then we get
$$\|\C_*[\nabla A_1,\cdots,\nabla A_n, f]\|_{L^r(\R^d,w)}\lc\Big(\prod_{i=1}^{n}\|\nabla A_i\|_{L^{q_i}(\R^d,w)}\Big)\|f\|_{L^\infty(\R^d)}.$$
\end{prop}

\begin{proof}
Let $\vphi$ be a $C_c^\infty$ function which is supported in $\{x\in\R^d: |x|<1/4\}$, $\vphi(x)=1$ if $|x|<1/8$ and $\int_{\R^d}\vphi(x)dx=1$. Set $\vphi_\eps(x)=\eps^{-d}\vphi(\eps^{-1}x)$.
It is easy to see that $\vphi_\eps*\C[\nabla A_1,\cdots,\nabla A_n, f](x)$ is bounded by $M(\C[\nabla A_1,\cdots,\nabla A_n, f])(x)$. By the weighted boundedness of the Hardy-Littlewood maximal operator $M$ and Proposition \ref{e:16unwrgeq1}, we may get that
\Bes
\|M(\C[\nabla A_1,\cdots,\nabla A_n, f])\|_{L^r(\R^d,w)}\lc\Big(\prod_{i=1}^{n}\|\nabla A_i\|_{L^{q_i}(\R^d,w)}\Big)\|f\|_{L^\infty(\R^d,w)}.
\Ees
So, to complete the proof, it is suffice to show that the following difference
$$\C_\eps[\nabla A_1,\cdots,\nabla A_n, f](x)-\vphi_\eps*\C[\nabla A_1,\cdots,\nabla A_n, f](x)$$
is controlled uniformly in $\eps$ by a function which is bounded from $L^{q_1}(\R^d,w)\times\cdots\times L^{q_n}(\R^d,w)\times L^\infty(\R^d,w)$ to $L^r(\R^d,w)$. We write the difference in the above equality as follows
\Bes
\begin{split}
\int_{\R^d}\vphi_\eps(z)&\Big[\int_{|x-y|>\eps}\Big(\mathfrak{K}(x,y)-\mathfrak{K}(x-z,y)\Big)f(y)dy\Big]dz\\
&+\int_{\R^d}\vphi_\eps(z)\Big[\pv\int_{|x-y|<\eps}\mathfrak{K}(x-z,y)f(y)dy\Big]dz=:P_\eps(x)+Q_\eps(x),
\end{split}
\Ees
where $\mathfrak{K}(x,y):=K(x-y)\prod_{i=1}^n\fr{A_i(x)-A_i(y)}{|x-y|}$. Now we first give an estimate of $Q_\eps(x)$. By the Fubini theorem,

\Bes
\begin{split}
Q_\eps(x)&=\int_{\R^d}\vphi_\eps(x-z)\Big[\pv\int_{|x-y|<\eps}\mathfrak{K}(z,y)f(y)dy\Big]dz\\
&=\int_{|x-y|<\eps}\Big[\pv\int_{\R^d}\vphi_\eps(x-z)\mathfrak{K}(z,y)dz\Big]f(y)dy.
\end{split}
\Ees

Notice that $|x-y|<\eps$ and $|x-z|<\fr{\eps}{4}$. For each $i=1,\cdots,n$, define $\tilde{A}_i(\cdot)=A_i(\cdot)\chi_{\{|\cdot-x|<\eps\}}$. The all $A_i$ in the above inequality can be replaced by $\tilde{A}_i$. Choose $\fr{1}{\tilde{r}}=(\sum_{i=1}^n\fr{1}{\tilde{q}_{i}})+\fr{1}{\tilde{p}}$ such that $1<\tilde{r}<+\infty$, $1\leq\tilde{q}_i<q_i<\infty$ for all $i=1,\cdots,n$, $1<\tilde{p}<\infty$.
Then by the multilinear boundedness properties of $\C[\nabla A_1,\cdots,\nabla A_n,f](x)$, we may continue to give an estimate of $Q_\eps(x)$ as follows
\Bes
\begin{split}
|Q_\eps(x)|&\lc\|f\|_{L^\infty(\R^d)}\eps^{\fr{d}{\tilde{r}'}}\prod_{i=1}^n\|\nabla \tilde{A}_i\|_{L^{\tilde{q}_i}(\R^d)}\|\vphi_\eps\|_{L^{\tilde{p}}(\R^d)}\\
&\leq\|f\|_{L^\infty(\R^d)}\Big(\prod_{i=1}^n\fr{1}{\eps^d}\int_{|x-z|<\eps}|\nabla A_i(z)|^{\tilde{q}_i}dz\Big)^{\fr{1}{\tilde{q}_i}}\eps^{\fr{d}{\tilde{r}'}+\sum_{i=1}^n\fr{d}{\tilde{q}_i}-d\fr{\tilde{p}-1}{\tilde{p}}}\|\vphi\|_{L^{\tilde{p}}(\R^d)}\\
&\lc\|f\|_{L^\infty(\R^d,w)}\prod_{i=1}^n M_{\tilde{q}_i}(\nabla A_i)(x).
\end{split}
\Ees
As we have done in the proof of Lemma \ref{l:mw}, we see that $M_{\tilde{q_i}}$ maps $L^{q_i}(\R^d,w)$ to itself for $w\in A_{q_i/d}(\R^d)$. Then by using the H\"older inequality, we may get that
\Bes
\|\sup_\eps|Q_\eps|\|_{L^r(\R^d,w)}\lc\Big(\prod_{i=1}^{n}\|\nabla A_i\|_{L^{q_i}(\R^d,w)}\Big)\|f\|_{L^\infty(\R^d,w)}.
\Ees

Next we turn to $P_\eps(x)$. Write
\Bes
\begin{split}
\mathfrak{K}(x,&y)-\mathfrak{K}(x-z,y)\\
&=\Big(\fr{K(x-y)}{|x-y|^n}-\fr{K(x-z-y)}{|x-z-y|^n}\Big){\prod_{i=1}^n(A_i(x)-A_i(y))}\\
&\quad+\fr{K(x-z-y)}{|x-z-y|^n}\Big[\prod_{i=1}^n(A_i(x)-A_i(y))-\prod_{i=1}^n(A_i(x-z)-A_i(y))\Big]\\
&=:I+II.
\end{split}
\Ees

We consider the term $I$. Notice that $|x-y|>\eps$ and $|z|<\fr{1}{4}\eps$, then $|x-y|\approx |x-z-y|$. By the regularity condition \eqref{e:12kr} and \eqref{e:16adiff}, we get that
\Bes
|I|\lc\fr{\eps^\del}{|x-y|^{d+\del}}\prod_{i=1}^d[M(\nabla A_i)(x)+T(\nabla A_i)(y)].
\Ees

Consider the term $II$. We may split $II$ into $n$ terms and use \eqref{e:16adiff},
\Bes
\begin{split}
|II|&\lc\fr{1}{|x-y|^{d+n}}\Big|\sum_{i=1}^n[A_i(x)-A_i(x-z)]\prod_{k=1}^{i-1}[A_k(x-z)-A_i(y)]\prod_{k=i+1}^{n}[A_k(x)-A_k(y)]\Big|\\
&\lc\fr{\eps}{|x-y|^{d+1}}\prod_{i=1}^n[M(\nabla A_i)(x)+T(\nabla A_i)(x-z)+T(\nabla A_i)(y)].
\end{split}
\Ees
Combining the estimates of $I$ and $II$, we get
\Bes
\begin{split}
|\mathfrak{K}(x,&y)-\mathfrak{K}(x-z,y)|\\
&\lc\fr{\eps^\del}{|x-y|^{d+\del}}\prod_{i=1}^n[M(\nabla A_i)(x)+T(\nabla A_i)(x-z)+T(\nabla A_i)(y)]\\
&\lc\fr{\eps^\del}{|x-y|^{d+\del}}\sum_{\N_1^n}\prod_{i\in N_1}M(\nabla A_i)(x)\prod_{i\in N_2}T(\nabla A_i)(x-z)\prod_{i\in N_3}T(\nabla A_i)(y),
\end{split}
\Ees
where in the last inequality we divide $\N_{1}^n=N_1\cup N_2\cup N_3$ with $N_1$, $N_2$, $N_3$ non intersecting each other. Plugging the above estimate into $P_\eps(x)$, we get that
\Bes
\begin{split}
|P_\eps(x)|&\lc\sum_{\N_1^n}\big[\prod_{i\in N_1}M(\nabla A_i)(x)\big]\int_{\R^d}\int_{|x-y|>\eps}\fr{\eps^\del}{|x-y|^{d+\del}}\big[\prod_{i\in N_3}T(\nabla A_i)(y)\big]f(y)dy\\
&\quad\quad\times\vphi_\eps(z)\big[\prod_{i\in N_2}T(\nabla A_i)(x-z)\big]dz\\
&\lc\|f\|_{L^\infty(\R^d)}\sum_{\N_1^n}\big[\prod_{i\in N_1}M(\nabla A_i)(x)\big]M\big[\prod_{i\in N_2}T(\nabla A_i)\big](x)\cdot M\big[\prod_{i\in N_3}T(\nabla A_i)\big](x).
\end{split}
\Ees
Now using the H\"older inequality and the fact that $M$, $T$ are bounded on $L^q(\R^d,w)$ for $1<q<\infty$, we get that
\Bes
\|\sup_{\eps}|P_\eps|\|_{L^r(\R^d,w)}\lc\Big(\prod_{i=1}^{n}\|\nabla A_i\|_{L^{q_i}(\R^d,w)}\Big)\|f\|_{L^\infty(\R^d,w)},
\Ees
which completes the proof.
\end{proof}

\subsection{Case: all $q_i$s are larger than $d$}\label{s:1223}\quad
\vskip0.24cm
In this subsection, we consider the case $d/(d+n)\leq r<\infty$ and $d\leq q_1,\cdots,q_n\leq\infty$. Without loss of generality, we assume that the first $q_1,\cdots,q_l>d$ and $q_{l+1},\cdots,q_n=d$ with $0\leq l\leq n$. Here and in the following, when $l=0$, we mean all $q_1=\cdots=q_n=d$. The proof of the case $p=\infty$ is slight different from that of $1\leq p<\infty$. So we shall give two propositions below. Let us see the case $1\leq p<\infty$ firstly and we emphasize in the proof where it doesn't work for $p=\infty$.
\begin{prop}\label{p:12qibigd}
Let $\fr{1}{r}=\big(\sum_{i=1}^n\fr{1}{q_i}\big)+\fr{1}{p}$, $\fr{d}{d+n}\leq r<\infty$, $d<q_1,\cdots,q_l\leq\infty$ and $q_{l+1} =\cdots=q_n=d$ with $0\leq l\leq n$, $1\leq p<\infty$. Suppose that $w\in \big(\bigcap_{i=1}^nA_{\max\{\fr{q_i}{d},1\}}(\R^d)\big)\cap {A_p}(\R^d)$. Then
\Be\label{e:12weakqgeqd}
\begin{split}
\|\C_*[\nabla A_1,&\cdots,\nabla A_n, f]\|_{L^{r,\infty}(\R^d,w)}\\
&\lc\Big(\prod_{i=1}^l\|\nabla A_i\|_{L^{q_i}(\R^d,w)}\Big)\Big(\prod_{i=l+1}^{n}\|\nabla A_i\|_{L^{d,1}(\R^d,w)}\Big)\|f\|_{L^p(\R^d,w)},
\end{split}
\Ee
where $L^{d,1}(\R^d,w)$ is the weighted Lorentz space.
\end{prop}

\begin{proof}
By the dense limiting argument and scaling argument, it is sufficient to prove that when $A_i$ ($i=1,\cdots,n$) and $f$ are $C^\infty$ functions with compact supports, $$\|\nabla A_i\|_{L^{q_i}(\R^d,w)}=\|\nabla A_j\|_{L^{d,1}(\R^d,w)}=\|f\|_{L^p(\R^d,w)}=1,$$
for $i=1,\cdots,l$ and $j=l+1,\cdots,n$, the following inequality
\Bes
w(\{x\in\R^d:\C_*[\nabla A_1,\cdots,\nabla A_n, f](x)>\lam\})\lc\lam^{-r}
\Ees
holds for any $\lam>0$.
Fix $\lam>0$. For convenience we set
\Be\label{e:12elam}
E_\lam=\{x\in\R^d:\mathcal{C}_*[\nabla A_1,\cdots,\nabla A_n, f](x)>\lam\}.
\Ee

Our goal is to show $w(E_\lam)\lc\lam^{-r}.$
First assume that all $q_1,\cdots, q_l<\infty$. Once the proof in this situation is well understood, we can modify the proof to the other case that there exist some $q_i=\infty$ for $i=1,\cdots,l$. We shall show how to do this in the last part of the proof.
Define the {\it exceptional set}
\Bes
\begin{split}
J_{i,\lam}=\big\{x\in\R^d:\M (\nabla A_i)(x)>\lam^{\fr{r}{q_i}}\big\}.
\end{split}
\Ees
for $i=1,\cdots,n$. Here it should be pointed out that the above definition is meaningless if $q_i=\infty$. Therefore we need to assume all $q_i<\infty$ firstly.
By Lemma \ref{l:mw} and Lemma \ref{l:11md}, $\M$ maps $L^p(\R^d,w)$ to itself for $p>d$ and maps $L^{d,1}(\R^d,w)$ to $L^{d,\infty}(\R^d,w)$, i.e.
\Be\label{e:11jlam}
\begin{split}
w(J_{i,\lam})\lc\lam^{-r}\|\nabla A_i\|^{q_i}_{L^{q_i}(\R^d,w)}&=\lam^{-r}, \ \ i=1,\cdots,l;\\
w(J_{j,\lam})\lc\lam^{-r}\|\nabla A_j\|^{d}_{L^{d,1}(\R^d,w)}&=\lam^{-r}, \ \ j=l+1,\cdots,n.
\end{split}
\Ee

Set $J_\lam=\cup_{i=1}^n J_{i,\lam}$. Since $w(x)dx$ satisfies the doubling property, we may choose an open set $G_\lam$ which satisfies the following conditions:
(1) $J_\lam\subset G_\lam$;
(2) $w(G_\lam)\lc w(J_\lam)$.
By the property \eqref{e:11jlam} of $J_{i,\lam}$, we see that $w(G_\lam)\lc\lam^{-r}$. Next making a Whitney decomposition of $G_\lam$ (see \eg \cite{Gra249}), we may obtain a family of disjoint dyadic cubes $\{Q_k\}_k$ such that
\begin{enumerate}[(i).]
\item \quad $G_\lam=\bigcup_{k=1}^\infty Q_k$;
\item \quad $\sqrt{d}\cdot l(Q_k)\leq dist(Q_k,(G_\lam)^c)\leq4\sqrt{d}\cdot l(Q_k).$
\end{enumerate}
With those properties (i) and (ii), for each $Q_k$, we may construct a larger cube $Q_k^*$ so that $Q_k\subset Q_k^*$, $Q_k^*$ is centered at $y_k$ and $y_k\in (G_\lam)^c$, $l(Q_k^*)\approx l(Q_k)$. By the property (ii) above, the distance between $Q_k$ and $(G_\lam)^c$ equals to $Cl(Q_k)$.
Therefore by the construction of $Q_k^*$ and $y_k$, one may get
\Be\label{e:12whitney}
dist(y_k,Q_k)\approx l(Q_k), \ \ w(Q_k^*)\approx w(Q_k).
\Ee

Now we come back to give an estimate of $w(E_\lam)$. Split $f$ into two parts $f=f_1+f_2$ where $f_1(x)=f(x)\chi_{(G_\lam)^c}(x)$ and $f_2(x)=f(x)\chi_{G_\lam}(x)$. By the definition of $J_\lam$, when restricted on $(G_\lam)^c$, $A_i$ is a Lipschitz function with $\|\nabla A_i\|_{L^\infty((G_\lam)^c)}\leq\lam^{\fr{r}{q_i}}$ for $i=1,\cdots,n$. Let $\tilde{A}_i$ represent the Lipschitz extension of $A_i$ from $(G_\lam)^c$ to $\R^d$ (see \cite[page 174, Theorem 3]{Ste70}) so that
for each $i=1,\cdots,n$,
$$\tilde{A}_i(y)=A_i(y)\ \ \text{if} \ y\in (G_\lam)^c;$$
$$\big|\tilde{A}_i(x)-\tilde{A}_i(y)\big|\leq\lam^{\fr{r}{q_i}}|x-y|\ \ \text{for all}\ x,y\in\R^d.$$

Since the operator $\mathcal{C}_*[\cdots,\cdot]$ is sub-multilinear, we split $E_\lam$ as three terms and give estimates as follows:
\Be\label{e:12spmu}
\begin{split}
w(\{x&\in\R^d: \mathcal{C}_*[\nabla A_1,\cdots,\nabla A_n,f](x)>\lambda\})\\
&\leq w(10G_\lam)+w\big(\{x\in (10G_\lam)^c:\mathcal{C}_*[\nabla A_1,\cdots,\nabla A_n,f_1](x)>\lambda/2\}\big)\\
&\ \ \ \ +w\big(\{x\in (10G_\lam)^c:\mathcal{C}_*[\nabla A_1,\cdots,\nabla A_n,f_2](x)>\lambda/2\}\big).
\end{split}
\Ee

The above first term  satisfies $w(10G_\lam)\lc\lam^{-r}$, which is our required estimate. In the following, we only consider the second terms. Notice that we only need to consider $x\in(10G_\lam)^c$. By the definition of $f_1$, it is not difficulty to see that
$$\mathcal{C}_*[\nabla A_1,\cdots,\nabla A_n,f_1](x)=\mathcal{C}_*[\nabla\tilde{A}_1,\cdots,\nabla\tilde{A}_n,f_1](x).$$
With this equality in hand, Proposition \ref{p:12strongr} ($1\leq p<\infty$) implies
\Be\label{e:12qbigd}
\begin{split}
w\big(\big\{x&\in (10G_\lam)^c: \mathcal{C}_*[\nabla A_1,\cdots,\nabla A_n,f_1](x)>{\lam}/{2}\big\}\big)\\
&=w\big(\big\{x\in (10G_\lam)^c: \mathcal{C}_*[\nabla \tilde{A}_1,\cdots,\nabla\tilde{A}_n,f_1](x)>{\lam}/{2}\big\}\big)\\
&\ \ \lc \lam^{-p}\Big(\prod_{i=1}^n\|\nabla\tilde{A}_i\|^p_{L^\infty(\R^d,w)}\Big)\|f_1\|^p_{L^p(\R^d,w)}
\lc\lam^{-p+p\sum_{i=1}^n\fr{r}{q_i}}=\lam^{-r}.
\end{split}
\Ee
If $p=\infty$, the above method does not work.
We will show how to prove this kind of estimate in the next proposition.

Let us turn to $\mathcal{C}_*[\nabla A_1,\cdots,\nabla A_n, f_2](x)$. Recall $\N_i^j=\{i,i+1,\cdots,j\}$ and our construction of $G_\lam$, $y_k$, $Q_k$ and $Q_k^*$ above \eqref{e:12whitney}. Then by the property (i) of $\{Q_k\}_k$, we may write $f_2=\sum_k f\chi_{Q_k}$. Therefore we may get
$$\mathcal{C}_\eps[\nabla A_1,\cdots,\nabla A_n, f_2](x)=\sum_k\mathcal{C}_\eps[\nabla A_1,\cdots,\nabla A_n, f\chi_{Q_k}](x).$$
In the following we need to study carefully  $\prod_{i=1}^n\fr{{A}_i(x)-A_i(y)}{|x-y|}$. We will separate it into several terms and then give an estimate for each term.  Write
\Bes
\begin{split}
&\ \ \ \ \prod_{i=1}^n\fr{{A}_i(x)-A_i(y)}{|x-y|}\\&=\prod_{i=1}^n\Big(\fr{\tilde{A}_i(x)-\tilde{A}_i(y)}{|x-y|}+\fr{\tilde{A}_i(y)-\tilde{A}_i(y_k)}{|x-y|}+\fr{A_i(y_k)-A_i(y)}{|x-y|}\Big)\\
&=\sum\Big(\prod_{i\in N_1}\fr{\tilde{A}_i(x)-\tilde{A}_i(y)}{|x-y|}\Big)\Big(\prod_{i\in N_2}\fr{\tilde{A}_i(y)-\tilde{A}_i(y_k)}{|x-y|}\Big)\Big(\prod_{i\in N_3}\fr{{A}_i(y_k)-{A}_i(y)}{|x-y|}\Big)\\
&=I(x,y)+II(x,y,y_k),
\end{split}
\Ees
where in the third equality we divide $\N_{1}^n=N_1\cup N_2\cup N_3$ with $N_1$, $N_2$, $N_3$ non intersecting each other; and $I(x,y)$, $II(x,y,y_k)$, are defined as follows
\Be\label{e:12axyqbigd}
\begin{split}
I(x,y)=&\prod_{i=1}^n\fr{\tilde{A}_i(x)-\tilde{A}_i(y)}{|x-y|},\\
II(x,y,y_k)=&\sum_{N_1\subsetneq\N_{1}^n}\Big[\prod_{i\in N_1}\fr{\tilde{A}_i(x)-\tilde{A}_i(y)}{|x-y|}\Big]\\
&\times\Big[\prod_{i\in N_2}\fr{\tilde{A}_i(y)-\tilde{A}_i(y_k)}{|x-y|}\Big]\Big[\prod_{i\in N_3}\fr{{A}_i(y_k)-{A}_i(y)}{|x-y|}\Big].
\end{split}
\Ee
By the above decomposition, we in fact write $\mathcal{C}_\eps[\nabla A_1,\cdots,\nabla A_n, f\chi_{Q_k}](x)$ into $3^n$ terms and separate these terms into two parts according $I$ and $II$.

\emph{Weighted estimate of $\mathcal{C}_*[\cdots,\cdot]$ related to $I$.}  This estimate is similar to \eqref{e:12qbigd}. In fact, in this case there is only one term $\C_*[\nabla\tilde{A}_1,\cdots,\nabla\tilde{A}_n,f_2]$. Then by Proposition \ref{p:12strongr} ($1\leq p<\infty$), we get
\Bes
\begin{split}
w\big(\big\{x\in& (10G_\lam)^c: \mathcal{C}_*[\nabla \tilde{A}_1,\cdots,\nabla\tilde{A}_n,f_2](x)>{\lam}/{2}\big\}\big)\\
&\lc \lam^{-p}\Big(\prod_{i=1}^n\|\nabla\tilde{A}_i\|^p_{L^\infty(\R^d,w)}\Big)\|f_2\|^p_{L^p(\R^d,w)}
\lc\lam^{-p+p\sum_{i=1}^n\fr{r}{q_i}}=\lam^{-r}.
\end{split}
\Ees
If $p=\infty$, the above argument may not work again.

\emph{Weighted estimate of $\mathcal{C}_*[\cdots,\cdot]$ related to $II$.}
It is sufficient to consider one term $\mathcal{C}_*[\cdots,\cdot]$ related to $II$ in which $N_1$ is a proper subset of $\N_1^n$.
In such a case, without loss of generality, we may suppose that $N_1=\{1,\cdots,v\}$, $N_2=\{v+1,\cdots,m\}$ and $N_3=\{m+1,\cdots, n\}$ with $0\leq v\leq m\leq n$ and $v<n$. Here if $v=0$, it means that $N_1=\emptyset$; if $v=m$, $N_2=\emptyset$; if $m=n$,  $N_3=\emptyset$. With these notation, it is easy to see that $N_1$ is a proper subset of $\N_1^n$.
By a slight abuse of notation, we still utilize $II(x,y,y_k)$ to represent one term related to $N_1$, $N_2$ and $N_3$ in \eqref{e:12axyqbigd} and utilize $H_{II}(x)$ to represent $\mathcal{C}_*[\cdots,\cdot]$ related to ${II}(x,y,y_k)$, i.e.
$$H_{II}(x)=\sup_{\eps>0}\Big|\sum_k\int_{|x-y|>\eps}K(x-y)II(x,y,y_k)f(y)\chi_{Q_k}(y)dy\Big|.$$
Notice that $y_k\in(G_\lam)^c$, thus $y_k\in (J_{i,\lam})^c$. Therefore we obtain that
\Bes
\M(\nabla A_i)(y_k)\leq\lam^{\fr{r}{q_i}}, \ \text{for $i=m+1,\cdots,n$.}
\Ees
With the above fact and $\tilde{A}_i$ is a Lipschitz function with bound $\lam^{r/q_i}$ for $i=1,\cdots,m$, we get
\Bes
\begin{split}
|II(x,y,y_k)|&\lc\lam^{\sum_{i=1}^m\fr{r}{q_i}}\fr{|y-y_k|^{n-v}}{|x-y|^{n-v}}\prod_{i=m+1}^n\M(\nabla A_i)(y_k)\\
&\lc\lam^{\sum_{i=1}^n\fr{r}{q_i}}\fr{|y-y_k|^{n-v}}{|x-y|^{n-v}}.
\end{split}
\Ees
Since it is sufficient to  consider $x\in(10G_\lam)^c$, then for $y\in Q_k$,  $|x-y|\geq 2l(Q_k)\approx|y-y_k|$ by \eqref{e:12whitney}. Combining the above discussion with \eqref{e:12kb}, we obtain
\Be\label{e:12qgeqdinfty}
\begin{split}
H_{II}(x)&\leq\sum_k\int_{Q_k}|K(x-y)|\cdot|II(x,y,y_k)|\cdot|f(y)|dy\\
&\lc\lam^{\sum_{i=1}^n\fr{r}{q_i}}\sum_k\int_{Q_k}\fr{l(Q_k)^{n-v}}{[l(Q_k)+|x-y|]^{d+n-v}}|f(y)|dy
=\lam^{\sum_{i=1}^n\fr{r}{q_i}}T_{n-v}f(x)
\end{split}
\Ee
where $T_{n-v}$ is defined in Lemma \ref{l:12disq1infty}.
Applying the Chebyshev inequality with the above estimate, and utilizing Lemma \ref{l:12disq1infty} (notice  that $n-v\geq1$ because $N_1$ is a proper set of $\N_1^n$), we finally get
\Bes
w(\{x\in(10G_\lam)^c: H_{II}(x)>\lam\})\leq\lam^{-p+\sum_{i=1}^n\fr{rp}{q_i}}\|T_{n-v}f\|^p_{L^p(\R^d,w)}\lc\lam^{-r}\|f\|^p_{L^p(\R^d,w)}.
\Ees
Hence we finish the proof of the term $II$. If $p=\infty$, the above last argument may not work and some different discussion should be involved, see the proof in the next proposition.

Finally, we show how to modify our proof  here to the case $q_i=\infty$ for some $i=1,\cdots,l$. We may assume that only $q_1=\cdots=q_u=\infty$ with $1\leq u\leq l$. Thus $A_1$, $\cdots$, $A_u$ are Lipschitz functions which in fact are nice functions. Then we just fix $A_1, \cdots, A_u$ in the rest of the proof. We only make a construction of {\it exceptional set\/} for $A_{u+1}, \cdots, A_n$ and study $\prod_{i={u+1}}^n\fr{A_i(x)-A_i(y)}{|x-y|}$ by using the same way as we have done previously.
After that utilizing $A_1$, $\cdots$, $A_u$ are Lipschitz functions to deal with all estimates involved with $A_1,\cdots,A_u$,
we could obtain our required bound.
\end{proof}
\begin{prop}\label{p:12qleqinfty}
Let $\fr{1}{r}=\big(\sum_{i=1}^n\fr{1}{q_i}\big)+\fr{1}{p}$, $\fr{d}{d+n}\leq r<\infty$, $d<q_1,\cdots,q_l\leq\infty$ and $q_{l+1}$, $\cdots$, $q_n=d$ with $0\leq l\leq n$, $p=\infty$. Suppose that $w\in \bigcap_{i=1}^nA_{\max\{\fr{q_i}{d},1\}}(\R^d)$. Then
\Bes
\begin{split}
\|\C_*[\nabla A_1,&\cdots,\nabla A_n, f]\|_{L^{r,\infty}(\R^d,w)}\\
&\lc\Big(\prod_{i=1}^l\|\nabla A_i\|_{L^{q_i}(\R^d,w)}\Big)\Big(\prod_{i=l+1}^{n}\|\nabla A_i\|_{L^{d,1}(\R^d,w)}\Big)\|f\|_{L^\infty(\R^d,w)},
\end{split}
\Ees
where $L^{d,1}(\R^d,w)$ is the weighted Lorentz space.
\end{prop}
\begin{proof}
The proof here is similar to that of Proposition \ref{p:12qibigd}. So we shall be brief and  only indicate necessary modifications here. Proceeding the proof in Proposition \ref{p:12qibigd}, there are four different arguments.

The first one is that when we choose the set $E_\lam$, we set
$$E_\lam=\{x\in\R^d:\mathcal{C}_*[\nabla A_1,\cdots,\nabla A_n, f](x)>C_0\lam\},$$
where $C_0$ is a constant determined later. Our goal is to show $w(E_\lam)\lc\lam^{-r}$. We split $E_\lam$ as several terms and give estimates as follows:
\Be\label{e:16pinfty}
\begin{split}
w(\{x&\in\R^d: \mathcal{C}_*[\nabla A_1,\cdots,\nabla A_n,f](x)>C_0\lambda\})\\
&\leq w(10G_\lam)+w\big(\{x\in (10G_\lam)^c:\mathcal{C}_*[\nabla A_1,\cdots,\nabla A_n,f_1](x)>C_0\lambda/2\}\big)\\
&\qquad \ \ \ +w\big(\{x\in (10G_\lam)^c:\mathcal{C}_*[\nabla A_1,\cdots,\nabla A_n,f_2](x)>C_0\lambda/2\}\big).
\end{split}
\Ee
The first term above satisfies $w(10G_\lam)\lc\lam^{-r}$, so it is sufficient to consider the second and third terms. Thus we only need to consider $x\in(10G_\lam)^c$.

The second difference is the estimate related to the second term in \eqref{e:16pinfty}. Here we choose $\tilde{r}$, $\tilde{q}_1$, $\cdots$, $\tilde{q}_n$, such that $1<\tilde{r}<\infty$, $q_1<\tilde{q}_1<\infty$, $\cdots$, $q_n<\tilde{q}_n<\infty$, $d<\tilde{q}_1,\cdots,\tilde{q}_n$ and $\fr{1}{\tilde{r}}=\sum_{i=1}^n\fr{1}{\tilde{q}_i}$. Utilize Proposition \ref{p:16maxinf} with those above $\tilde{r}$, $\tilde{q}_1$, $\cdots$, $\tilde{q}_n$ and $A_i$ is a Lipschitz function on $(G_\lam)^c$ with Lipschitz bound $\lam^{\fr{r}{q_i}}$ for $i=1,\cdots,n$ , we may obtain
\Bes
\begin{split}
w&\big(\big\{x\in (10G_\lam)^c: \mathcal{C}_*[\nabla A_1,\cdots,\nabla A_n,f_1](x)>{C_0\lam}/{2}\big\}\big)\\
&\leq w\big(\big\{x\in (G_\lam)^c: \mathcal{C}_*[\nabla ({A}_1\chi_{(G_\lam)^c}),\cdots,\nabla({A}_n\chi_{(G_\lam)^c}),f_1](x)>{C_0\lam}/{2}\big\}\big)\\
&\lc \lam^{-\tilde{r}}\Big(\prod_{i=1}^n\|\nabla({A}_i\chi_{(G_\lam)^c})\|^{\tilde{r}}_{L^{\tilde{q}_i}(\R^d,w)}\Big)\|f_1\|^{\tilde{r}}_{L^\infty(\R^d,w)}\\
&\lc
\lam^{-\tilde{r}}\Big(\prod_{i=1}^n\|\nabla{A}_i\|^{(\tilde{q}_i-q_i)\fr{\tilde{r}}{\tilde{q}_i}}_{L^{\infty}({(G_\lam)^c})}\Big)
\Big(\prod_{i=1}^n \|\nabla A_i\|_{L^{q_i}(\R^d,w)}^{\fr{q_i}{\tilde{q}_i}\tilde{r}}\Big)\|f_1\|^{\tilde{r}}_{L^\infty(\R^d,w)}\\
&\lc\lam^{-\tilde{r}+\tilde{r}\big(\sum_{i=1}^n\fr{r}{q_i}\big)-{r}\big(\sum_{i=1}^n\fr{\tilde{r}}{\tilde{q}_i}\big)}=\lam^{-r}.
\end{split}
\Ees

Next consider the estimate related to the third term in \eqref{e:16pinfty}. As done in the proof of Proposition \ref{p:12qibigd}, we divide $\mathcal{C}_\eps[\nabla A_1,\cdots,\nabla A_n, f_2](x)$ into several terms and then separate these terms into two parts according $I$ and $II$ in \eqref{e:12axyqbigd}. So we get
\Bes
\begin{split}
&w\big(\{x\in (10G_\lam)^c:\mathcal{C}_*[\nabla A_1,\cdots,\nabla A_n,f_2](x)>C_0\lambda/2\}\big)\\
&\qquad\leq w\big(\big\{x\in (10G_\lam)^c: \mathcal{C}_*[\nabla \tilde{A}_1,\cdots,\nabla\tilde{A}_n,f_2](x)>{C_0\lam}/{4}\big\}\big)\\
&\qquad\qquad+ w\big(\big\{x\in (10G_\lam)^c: H_{II}(x)>{C_0\lam}/{4}\big\}\big).
\end{split}
\Ees

The third difference is the {\it weighted estimate of $\mathcal{C}_*[\cdots,\cdot]$ related to $I$\/}. Here we utilize Lemma \ref{p:12strongr} and the estimate $\|f_2\|_{L^1(\R^d,w)}\lc\|f\|_{L^\infty(\R^d,w)}w(G_\lam)\lc\lam^{-r}$ to get
\Bes
\begin{split}
w\big(\big\{x\in& (10G_\lam)^c: \mathcal{C}_*[\nabla \tilde{A}_1,\cdots,\nabla\tilde{A}_n,f_2](x)>{C_0\lam}/{4}\big\}\big)\\
&\lc \lam^{-1}\Big(\prod_{i=1}^n\|\nabla\tilde{A}_i\|_{L^\infty(\R^d)}\Big)\|f_2\|_{L^1(\R^d,w)}
\lc\lam^{-1+\big(\sum_{i=1}^n\fr{r}{q_i}\big)-r}=\lam^{-r}.
\end{split}
\Ees

The fourth difference is the {\it weighted estimate of $\mathcal{C}_*[\cdots,\cdot]$ related to $II$\/}. We shall prove that
\Be\label{e:12Hiiemp}\big\{x\in (10G_\lam)^c: H_{II}(x)>{C_0\lam}/{4}\big\}=\emptyset.
\Ee
In fact, by \eqref{e:12qgeqdinfty} and Lemma \ref{l:12disq1infty} with $q=\infty$, we get for any $x\in(10G_\lam)^c$,
\Bes
\begin{split}
H_{II}(x)\leq C_{d}\lam^{\sum_{i=1}^n\fr{r}{q_i}}\|f\|_{L^\infty(\R^d,w)}=C_d\lam.
\end{split}
\Ees
If we choose $C_0>4C_d$, we get \eqref{e:12Hiiemp}. So we complete the proof.
\end{proof}
\vskip0.24cm

\subsection{Case: some $q_i$s are smaller than $d$ and some are not}\label{s:1225}\quad
\vskip0.24cm
In this subsection, we consider the case: $d/(d+n)\leq r<\infty$ with at least one $q_i< d$, $1\leq p\leq\infty$. By our condition, the weight $w$ satisfies $w\in \big(\bigcap_{i=1}^nA_{\max\{\fr{q_i}{d},1\}}(\R^d)\big)\cap {A_p}(\R^d)=A_1(\R^d)$.
Without loss of generality, we may suppose that $d\leq q_1,\cdots,q_l\leq\infty$ and $1\leq q_{l+1},\cdots,q_n< d$ with $0\leq l< n$. If $l=0$, it means that all $q_1,\cdots, q_n \in [1,d)$. Also we suppose  that
$q_{1}=\cdots=q_{k}=d$ and $d<q_{k+1},\cdots,q_l\leq\infty$  with $0\leq k\leq l$. If $k=0$, we mean that there is no index in $q_1,\cdots,q_l$ equals to $d$, i.e. $d<q_1,\cdots,q_l\leq\infty$; if $k=l$, we mean that $q_1=\cdots=q_l=d$. Since the proof of $p=\infty$ is a little different from that of $1\leq p<\infty$, we shall give two propositions.

\begin{prop}\label{p:12qgeqleqd}
Suppose $w\in A_1(\R^d)$. Let $\fr{1}{r}=\big(\sum_{i=1}^n\fr{1}{q_i}\big)+\fr{1}{p}$, $\fr{d}{d+n}\leq r<\infty$, $q_{1}=\cdots=q_{k}=d$, $d< q_{k+1},\cdots,q_l\leq\infty$ and $1\leq q_{l+1},\cdots,q_n< d$ with $0\leq k\leq l$ and $0\leq l<n$, $1\leq p<\infty$. Then
\Bes
\begin{split}
\|\C_*[\nabla A_1,&\cdots,\nabla A_n, f]\|_{L^{r,\infty}(\R^d,w)}\\
&\lc\Big(\prod_{i=1}^k\|\nabla A_i\|_{L^{d,1}(\R^d,w)}\Big)\Big(\prod_{i=k+1}^n\|\nabla A_i\|_{L^{q_i}(\R^d,w)}\Big)\|f\|_{L^p(\R^d,w)},
\end{split}
\Ees
where $L^{d,1}(\R^d,w)$ is the weighted Lorentz space.
\end{prop}
\begin{proof}
We need  to prove that for any $\lam>0$, the following inequality holds
\Be\label{e:12qleqd23}
\begin{split}
w(\{&x\in\R^d:\C_*[\nabla A_1,\cdots,\nabla A_n, f](x)>\lam\})\\
&\lc\lam^{-r}\Big(\prod_{i=1}^k\|\nabla A_i\|^r_{L^{d,1}(\R^d,w)}\Big)\Big(\prod_{i=k+1}^n\|\nabla A_i\|^r_{L^{q_i}(\R^d,w)}\Big)\|f\|^r_{L^p(\R^d,w)}.
\end{split}
\Ee
By the standard dense and scaling argument, it is sufficient to consider that each $A_i$ ($i=1,\cdots,n$) and $f$ as smooth functions with compact supports and
\Bes
\begin{split}
&\|\nabla A_1\|_{L^{d,1}(\R^d,w)}=\cdots=\|\nabla A_{k}\|_{L^{d,1}(\R^d,w)}=1;\\
&\|\nabla A_i\|_{L^{q_i}(\R^d,w)}=1\ \ \text{for $i=k+1,\cdots,n$, and } \|f\|_{L^p(\R^d,w)}=1.
\end{split}
\Ees

As done in the proof of Proposition \ref{p:12qibigd}, we first suppose that all $q_{k+1},\cdots, q_l<\infty$ since the other case is easy and we will show lastly how to modify the proof to the case that there exist $q_i=\infty$ for some $i=k+1,\cdots,l$. Fix $\lam>0$ and set $$E_\lam=\{x\in\R^d:\mathcal{C}_*[\nabla A_1,\cdots,\nabla A_n, f](x)>\lam\}.$$
Our goal is to show $w(E_\lam)\lc\lam^{-r}$. The main idea is to construct some {\it exceptional set\/} such that the $w$ measure of {\it exceptional set\/} is bounded by $\lam^{-r}$, which is our required estimate. At the same time on the complementary set of {\it exceptional set\/} these functions $A_i$s should be Lipschitz functions with bound $\lam^{\fr{r}{q_i}}$ for each $i=1,\cdots,n$. The constructions of {\it exceptional sets\/} are different between $d\leq q_i<\infty$ and  $1\leq q_i<d$. Now we begin our constructions of some {\it exceptional sets\/}.

\subsubsection*{Step 1: Exceptional set related to $q_1,\cdots,q_l$} Define the {\it exceptional set} {for $i=1,\cdots,l$}
\Bes
\begin{split}
J_{i,\lam}=\big\{x\in\R^d:\M (\nabla A_i)(x)>\lam^{\fr{r}{q_i}}\big\};\ \ J_\lam=\cup_{i=1}^l J_{i,\lam}.
\end{split}
\Ees
Since $w\in A_1(\R^d)$, by Lemma \ref{l:mw} and Lemma \ref{l:11md}, $\M$ maps $L^p(\R^d,w)$ to itself for $p>d$ and maps $L^{d,1}(\R^d,w)$ to $L^{d,\infty}(\R^d,w)$, i.e.
\Be\label{e:11jlams3}
\begin{split}
w(J_{i,\lam})&\lc\lam^{-r}\|\nabla A_i\|^{d}_{L^{d,1}(\R^d,w)}=\lam^{-r}, \ \ i=1,\cdots,k;\\
w(J_{i,\lam})&\lc\lam^{-r}\|\nabla A_i\|^{q_i}_{L^{q_i}(\R^d,w)}=\lam^{-r}, \ \ i=k+1,\cdots,l.
\end{split}
\Ee
So we obtain that $w(J_\lam)\lc\lam^{-r}$.

\subsubsection*{{Step 2: Calder\'on-Zygmund decomposition.}}By the formula given in \cite[page 125, (17)]{Ste70}, for each $A_i$, $i=l+1,\cdots,n$, we may write
\Bes
A_i(x)=\sum_{j=1}^dC_d\int_{\R^d}\fr{x_j-y_j}{|x-y|^d}\pari_j A_i(y)dy=:\sum_{j=1}^dA_{i,j}(x).
\Ees
Notice that $w\in A_1(\R^d)$ in this case. For each $|\pari_jA_i|^{q_i}\in L^1(\R^d,w)$ with $j=1,\cdots, d$ and $i=l+1,\cdots,n$, making a Calder\'on-Zygmund decomposition at level $\lam^r$,
we may have the following conclusions (see \eg \cite[page 413, Theorem 3.5]{GR85}):
\begin{enumerate}[\quad (cz-i)]
\rm\item $\pari_j A_i=g_{j,i}+b_{j,i}$, $\|g_{j,i}\|_{L^\infty(\R^d)}\lc\lambda^{\fr{r}{q_i}}$, $\|g_{j,i}\|_{L^{q_i}(\R^d,w)}\lc\|\pari_jA_i\|_{L^{q_i}(\R^d,w)}$;
\item $b_{j,i}=\sum_{Q\in \mathcal{Q}_{j,i}} b_{j,i,Q}$, $\supp\  b_{j,i,Q}\subset Q$, where $\mathcal{Q}_{j,i}$ is a countable set of disjoint dyadic cubes;
\item Let $E_{j,i}=\bigcup_{Q\in \mathcal{Q}_{j,i}} Q$, then $w(E_{j,i})\lc {\lambda^{-{r}}}\|\pari_jA_i\|^{q_i}_{L^{q_i}(\R^d,w)}$;
\item $\int b_{j,i,Q}(y)dy=0$ for each $Q\in \mathcal{Q}_{j,i}$,  the unweighted estimate $\|b_{j,i,Q}\|^{q_i}_{L^{q_i}(\R^d)}\lc{\lambda^{{r}}}|Q|$  and the weighted estimate $\|b_{j,i}\|_{L^{q_i}{(\R^d,w)}}\lc\|\pari_jA_i\|_{L^{q_i}(\R^d,w)}$ hold.
\end{enumerate}

We shall split $A_{i,j}$ into two parts according the above Calder\'on-Zygmund decomposition (cz-i):
\Bes
\begin{split}
A^g_{i,j}(x)&=C_d\int_{\R^d}\fr{x_j-y_j}{|x-y|^d}g_{j,i}(y)dy;\\
A^b_{i,j}(x)&=C_d\int_{\R^d}\fr{x_j-y_j}{|x-y|^d}b_{j,i}(y)dy.
\end{split}
\Ees
Define the \emph{exceptional set}$B_\lam=\cup_{i=l+1}^n\cup_{j=1}^dE_{j,i}$. Then by (cz-iii), we obtain
$w(B_\lam)\lc\lam^{-r}.$

\subsubsection*{Step 3: Exceptional set $D_\lam$.} Set $\fr{1}{s_i}=\fr{1}{q_i}-\fr{1}{d}$ for $i=l+1,\cdots,n$. Notice that $w$ belongs to $A_1(\R^d)$. Define the following \emph{exceptional set}
\Bes
D_{i,\lam}=\Big\{x\in\R^d:\mathfrak{M}_{w,s_i}(\nabla A_i)(x)>\lam^{\fr{r}{q_i}}\Big\}
\Ees
where the maximal operator $\mathfrak{M}_{w,s_i}$ is defined in the paragraph above Lemma \ref{l:16sobolev}. Denote $D_\lam=\cup_{i=l+1}^nD_{i,\lam}$. Then by Lemma \ref{l:12qleqd}, we get that
\Bes
w(D_{i,\lam})\lc\lam^{-r}\|\nabla A_i\|^{q_i}_{L^{q_i}(\R^d,w)}=\lam^{-r};\ \ w(D_\lam)\lc\lam^{-r}.
\Ees

\subsubsection*{Step 4: Exceptional set $F_\lam$.} For each $j=1,\cdots,d$, $i=l+1,\cdots,n$, define the functions
\Bes
\Delta^{j,i}(x)=\sum_{Q\in\mathcal{Q}_{j,i}}\fr{l(Q)}{[l(Q)+|x-y_Q|]^{d+1}}m(Q)
\Ees
where $y_Q$ is the center of $Q$. Define another \emph{exceptional set}
\Bes
F_{j,i,\lam}=\{x\in\R^d:\Delta^{j,i}(x)>1\},\ \ F_{\lam}=\cup_{j=1}^d\cup_{i=l+1}^n F_{j,i,\lam}.
\Ees
Notice that we have $w\in A_1(\R^d)$. We claim the following  property of $A_1$ weight: For any cube $Q$ and $\alp>1$, there exists a constant $C$ independent of $Q$ and $\alp$ such that
\Be\label{e:16alpdoub}
w(\alp Q)\leq C\alp^d w(Q).
\Ee
In fact by \eqref{e:16a1} in Definition \ref{d:16a1}, there exist a constant $C$ independent of $\alp$ and $Q$ such that
\Bes
\fr{1}{|\alp Q|}\int_{\alp Q}w(z)dz\leq C\ \text{ess inf}_{y\in Q}w(y)\leq C\fr{1}{|Q|}\int_{Q}w(y)dy,
\Ees
which immediately implies \eqref{e:16alpdoub}. In the following, by using the Chebyshev inequality in the first inequality, \eqref{e:16alpdoub} with $\alp=2^k$ in the last second inequality and (cz-iii) in the last inequality, we get
\Bes
\begin{split}
w(F_{j,i,\lam})&\leq\int_{\R^d}\Delta_{j,i,\lam}(x)w(x)dx\lc\sum_{Q\in\mathcal{Q}_{j,i}}\Big[\int_{\R^d}\fr{l(Q)^{d+1}}{[l(Q)+|x-y_Q|]^{d+1}}w(x)dx\Big]\\
&\leq\sum_{Q\in\mathcal{Q}_{j,i}}\Big[\int_{Q}w(x)dx+\sum_{k=1}^\infty\int_{2^kQ\setminus 2^{k-1}Q}\fr{l(Q)^{d+1}}{[l(Q)+|x-y_Q|]^{d+1}}w(x)dx\Big]\\
&\lc\sum_{Q\in\mathcal{Q}_{j,i}}\Big[w(Q)+\sum_{k=1}^\infty \fr{w(2^kQ)}{2^{k(d+1)}}\Big]\lc\sum_{Q\in\mathcal{Q}_{j,i}}w(Q)\lc\lam^{-r}.
\end{split}
\Ees
Therefore we obtain that $w(F_\lam)\lc\lam^{-r}$.
\subsubsection*{Step 5: Exceptional set $H_\lam$.} Define the {\it exceptional set\/} for $i=l+1,\cdots,n$, $j=1,\cdots,d$,
\Bes
H_{i,j,\lam}=\{x\in\R^d:\M(\nabla A^g_{i,j})(x)>\lam^{r/q_i}\},\ \ H_\lam=\cup_{i=l+1}^n\cup_{j=1}^d H_{i,j,\lam}.
\Ees
Notice that by the definition of $A^g_{i,j}$, for each $s=1,\cdots,d$, we get $$\F(\pari_{s}A_{i,j}^g)(\xi)=C\fr{\xi_s\xi_j}{|\xi|^2}\F(g_{j,i})(\xi)\Rightarrow\nabla A^g_{i,j}=CRR_jg_{j,i},$$
where $\F$ is the Fourier transform, $R_j$ is the Riesz transform and $R=(R_1,\cdots,R_d)$.
Recall $w\in A_1(\R^d)$. Since $R_j$ is weighted strong type $(q,q)$ for $1<q<\infty$, we get $\|\nabla A^g_{i,j}\|_{L^q(\R^d,w)}\lc\|g_{j,i}\|_{L^q(\R^d,w)}.$ Choose  $d<q<\infty$, by the Chebyshev inequality, Lemma \ref{l:mw} and (cz-i) in Step 2, we get
\Bes
\begin{split}
w(H_{i,j,\lam})&\lc\lam^{-\fr{qr}{q_i}}\int_{\R^d}[\M(\nabla A_{i,j}^g)(x)]^qw(x)dx\lc\lam^{-\fr{qr}{q_i}}\int_{\R^d}|\nabla A_{i,j}^g(x)|^qw(x)dx\\
&\lc\lam^{-\fr{qr}{q_i}}\int_{\R^d}|g_{j,i}(x)|^qw(x)dx\lc\lam^{-r}\int_{\R^d}|g_{j,i}(x)|^{q_i}w(x)dx\lc\lam^{-r}.
\end{split}
\Ees
Therefore we get $w(H_\lam)\lc\lam^{-r}$.

\subsubsection*{Step 6: Final exceptional set $G_\lam$} Based on the construction of $J_\lam, B_\lam, D_\lam, F_\lam, H_\lam$ in Step 1-5 and the fact $w$ satisfies the doubling property,
we choose an open set $G_\lam$ which satisfies the following conditions:
\begin{enumerate}[(1).]
\item\quad $\big(10J_\lam\cup 10B_\lam\cup 10D_\lam\cup 10F_\lam\cup 10H_\lam\big)\subset G_\lam$;
\item\quad $w(G_\lam)\lc w(J_\lam)+w(B_\lam)+w(D_\lam)+w(F_\lam)+w(H_\lam)$.
\end{enumerate}
Applying the previous weighted estimates of $J_\lam$, $B_{\lam}$, $D_\lam$, $F_\lam$ and $H_\lam$, we obtain that $w(G_\lam)\lc\lam^{-r}$. Next making a Whitney decomposition of $G_\lam$ (see \cite{Gra249}), we may obtain a family of disjoint dyadic cubes $\{Q_k\}_k$ such that
\begin{enumerate}[(i).]
\item \quad $G_\lam=\bigcup_{k=1}^\infty Q_k$;
\item \quad $\sqrt{d}\cdot l(Q_k)\leq dist(Q_k,(G_\lam)^c)\leq4\sqrt{d}\cdot l(Q_k).$
\end{enumerate}
By the property (ii) above, the distance between $Q_k$ and $(G_\lam)^c$ equals to $Cl(Q_k)$. For each $Q_k$ above, we could construct a larger cube $Q_k^*$ so that $Q_k\subset Q_k^*$, $Q_k^*$ is centered at $y_k$ and $y_k\in (G_\lam)^c$, $l(Q_k^*)\approx l(Q_k)$.
Therefore by the construction of $Q_k^*$ and $y_k$, we may get that
\Be\label{e:12whitney3}
dist(y_k,Q_k)\approx l(Q_k).
\Ee
\vskip 0.24cm
Clearly, the \emph{exceptional set} $G_\lam$ constructed in Step 6 satisfies that $w(G_\lam)\lc\lam^{-r}$. Below we will prove that these functions $A_i$s are Lipschitz functions on $(G_\lam)^c$.

\subsubsection*{Step 7: Lipschitz estimates of $A_i$ on $(G_\lam)^c$} Choose any $x,y\in(G_\lam)^c$. By the {\it exceptional set} $J_\lam$ constructed in Step 1, we see that for $i=1,\cdots,l$
\Be\label{e:12lipgeqd}
|A_i(x)-A_i(y)|\leq\lam^{\fr{r}{q_i}}|x-y|.
\Ee
In the following we only consider $i=l+1,\cdots,n$.
By the Calder\'on-Zygmund decomposition in Step 2, it is sufficient to prove that $A_{i,j}^g$ and $A_{i,j}^b$ satisfy Lipschitz estimates on $\big(G_\lam\big)^c$ for each $i=l+1,\cdots,n$ and $j=1,\cdots,d$. Firstly, it is easy to see that $A_{i,j}^g$ satisfies Lipschitz estimates by the construction of $H_\lam$ in Step 5. In fact, $x,y\in (G_\lam)^c$ implies that $x,y\in H^c_\lam$, we obtain that for $i=l+1,\cdots,n$, $j=1,\cdots,d$,
\Be\label{e:12lipg3}
|A_{i,j}^g(x)-A_{i,j}^g(y)|\leq\lam^{\fr{r}{q_i}}|x-y|.
\Ee

We devote to proving that $A_{i,j}^b$ is a Lipschitz function on $(G_\lam)^c$. Recall the Calder\'on-Zygmund decomposition properties (cz-ii), (cz-iii) and (cz-iv) in Step 2. For each $b_{j,i}=\sum_{Q\in \mathcal{Q}_{j,i}} b_{j,i,Q}$, $\supp\ b_{j,i,Q}\subset Q$, where $\mathcal{Q}_{j,i}$ is a countable set of disjoint dyadic cubes. Then for each $Q\in\mathcal{Q}_{j,i}$, we define
$$A_{i,j}^{b_Q}(x)=C_d\int_{\R^d}\fr{x_j-z_j}{|x-z|^d}b_{j,i,Q}(z)dz.$$
Now we fix a dyadic cube $Q\in\mathcal{Q}_{j,i}$. We are going to give a straight-forward Lipschitz estimate of $A_{i,j}^{b_Q}$. By the construction of $G_\lam$, we get that $x,y\in(10B_\lam)^c$, i.e. $x,y\in (10Q)^c$, therefore we obtain
$dist(x,Q)\geq \fr{9}{2}l(Q)$ and $dist(y,Q)\geq \fr{9}{2}l(Q)$.
 Let $z_Q$ be the center of $Q$. Without loss of generality, assume that $|x-z_Q|\leq|y-z_Q|$. Choose a point $Z\in\R^d$ such that
\Bes
\begin{split}
|x-Z|<100|x-y|;\ \ |y-Z|\leq100|x-y|;\ \ |X-z_Q|>\fr{2}{5}|x-z_Q|
\end{split}
\Ees
for any $X$ belongs to the polygonal with vertex $x,y,Z$.
We could draw a figure to show that such a point $Z$ exists under the condition that $dist(x,Q)>\fr{9}{2}l(Q)$ and $dist(y,Q)>\fr{9}{2}l(Q)$.
Now we write $A_{i,j}^{b_Q}(x)-A_{i,j}^{b_Q}(y)=A_{i,j}^{b_Q}(x)-A_{i,j}^{b_Q}(Z)+A_{i,j}^{b_Q}(Z)-A_{i,j}^{b_Q}(y).$
By using the mean value formula, we have
\Be\label{e:12Aijc}
\begin{split}
&A_{i,j}^{b_Q}(x)-A_{i,j}^{b_Q}(Z)=\int_0^1\inn{x-Z}{\nabla (A_{i,j}^{b_Q})(tx+(1-t)Z)}dt;\\
&A_{i,j}^{b_Q}(Z)-A_{i,j}^{b_Q}(y)=\int_0^1\inn{Z-y}{\nabla (A_{i,j}^{b_Q})(tZ+(1-t)y)}dt.
\end{split}
\Ee
For any $0\leq t\leq1$, the points $tx+(1-t)Z$ and $tZ+(1-t)y$ lie in the polygonal with vertex $x,y,Z$. Notice that $|x-z_Q|\geq5l(Q)$. Then by our choice of $Z$, we get
\Be\label{e:12xyZt}
\begin{split}
&|tx+(1-t)Z-z_Q|>\fr{2}{5}|x-z_Q|>2l(Q),\\
&|tZ+(1-t)y-z_Q|>\fr{2}{5}|x-z_Q|>2l(Q).
\end{split}
\Ee
We set $Z(t)$ equals to $tx+(1-t)Z$ or $tZ+(1-t)y$ and $K_j(x)=x_j/|x|^d$. Using the cancelation condition of $b_{j,i,Q}$, \eqref{e:12xyZt} and the unweighted estimate  in (cz-iv) of  Step 2, we get that
\Bes
\begin{split}
|\nabla (A_{i,j}^{b_Q})(Z(t))|&=\Big|\int_{\R^d}\Big[(\nabla K_j)(Z(t)-z)-(\nabla K_j)(Z(t)-z_Q)\Big]b_{j,i,Q}(z)dz\Big|\\
&\lc\fr{l(Q)}{[l(Q)+|x-z_Q|]^{d+1}}\|b_{j,i,Q}\|_{L^1(Q)}\lc\lam^{\fr{r}{q_i}}\fr{l(Q)}{[l(Q)+|x-z_Q|]^{d+1}}|Q|.
\end{split}
\Ees
Combining the above arguments with \eqref{e:12Aijc} and the construction of $Z$, we obtain
\Bes
\big|A_{i,j}^{b_Q}(x)-A_{i,j}^{b_Q}(y)\big|\lc\lam^{\fr{r}{q_i}}\fr{l(Q)}{[l(Q)+|x-z_Q|]^{d+1}}|Q||x-y|.
\Ees
Notice $x\in (G_\lam)^c$ implies that $x\in (F_\lam)^c$ in Step 3. Then we get that
\Be\label{e:12lipb}
\big|A_{i,j}^{b}(x)-A_{i,j}^{b}(y)\big|\lc\lam^{\fr{r}{q_i}}|x-y|\sum_{Q\in\mathcal{Q}_{j,i}}\fr{l(Q)}{[l(Q)+|x-z_Q|]^{d+1}}|Q|\leq\lam^{\fr{r}{q_i}}|x-y|.
\Ee

Therefore we conclude that the Lipschitz estimates in \eqref{e:12lipgeqd} for $i=1,\cdots,l$, good function \eqref{e:12lipg3} and bad function \eqref{e:12lipb} for $i=l+1,\cdots,n$, to obtain that for any $i=1,\cdots,n$, $x,y\in (G_\lam)^c$,
\Be\label{e:12qileqdlip3}
|A_i(x)-A_i(y)|\leq\lam^{\fr{r}{q_i}}|x-y|.
\Ee
\vskip 0.24cm

\subsubsection*{Step 8: Weighted estimate of $E_\lam$} We come back to give an estimate of $E_\lam$. Split $f$ into two parts $f=f_1+f_2$ where $f_1(x)=f(x)\chi_{(G_\lam)^c}(x)$ and $f_2(x)=f(x)\chi_{G_\lam}(x)$. By the Lipschitz estimate in \eqref{e:12qileqdlip3}, when restricted on $(G_\lam)^c$, $A_i$ is a Lipschitz function with $\|\nabla A_i\|_{L^\infty((G_\lam)^c)}\leq\lam^{\fr{r}{q_i}}$ for $i=1,\cdots,n$. Let $\tilde{A}_i$ represent the Lipschitz extension of $A_i$ from $(G_\lam)^c$ to $\R^d$ (see \cite[page 174, Theorem 3]{Ste70}) so that
for each $i=1,\cdots,n$,
\Be\label{e:12qileqde3}
\begin{split}
\tilde{A}_i(y)&=A_i(y)\ \ \text{if} \ y\in (G_\lam)^c;\\
\big|\tilde{A}_i(x)-\tilde{A}_i(y)\big|&\leq\lam^{\fr{r}{q_i}}|x-y|\ \ \text{for all}\ x,y\in\R^d.
\end{split}
\Ee

Since the operator $\mathcal{C}_*[\cdots,\cdot]$ is sub-multilinear, we split $E_\lam$ as three terms
\Bes
\begin{split}
w(\{x&\in\R^d: \mathcal{C}_*[\nabla A_1,\cdots,\nabla A_n,f](x)>\lambda\})\\
&\leq w(10G_\lam)+w\big(\{x\in (10G_\lam)^c:\mathcal{C}_*[\nabla A_1,\cdots,\nabla A_n,f_1](x)>\lambda/2\}\big)\\
&\ \ \ \ +w\big(\{x\in (10G_\lam)^c:\mathcal{C}_*[\nabla A_1,\cdots,\nabla A_n,f_2](x)>\lambda/2\}\big).
\end{split}
\Ees
The above first term satisfies $w(10G_\lam)\lc\lam^{-r}$, which is our required estimate. Below we consider the second term. We only consider $x\in(10G_\lam)^c$. By the definition of $f_1$,
$$\mathcal{C}_*[\nabla A_1,\cdots,\nabla A_n,f_1](x)=\mathcal{C}_*[\nabla\tilde{A}_1,\cdots,\nabla\tilde{A}_n,f_1](x).$$
Notice that $w\in A_1(\R^d)\subsetneq A_p(\R^d)$. Applying the above equality and Proposition \ref{p:12strongr} ($1\leq p<\infty$), we derive that
\Bes
\begin{split}
w\big(\big\{x&\in (10G_\lam)^c: \mathcal{C}_*[\nabla A_1,\cdots,\nabla A_n,f_1](x)>{\lam}/{2}\big\}\big)\\
&=w\big(\big\{x\in (10G_\lam)^c: \mathcal{C}_*[\nabla \tilde{A}_1,\cdots,\nabla\tilde{A}_n,f_1](x)>{\lam}/{2}\big\}\big)\\
&\ \ \lc \lam^{-p}\Big(\prod_{i=1}^n\|\nabla\tilde{A}_i\|^p_{L^\infty(\R^d,w)}\Big)\|f_1\|^p_{L^p(\R^d,w)}
\lc\lam^{-p+p\sum_{i=1}^n\fr{r}{q_i}}=\lam^{-r}.
\end{split}
\Ees
If $p=\infty$, the above argument does not work.

\subsubsection*{Step 9: Weighted estimate of $\mathcal{C}_*[\nabla A_1,\cdots,\nabla A_n, f_2](x)$}  Recall $\N_i^j=\{i,i+1,\cdots,j\}$ and the construction of $G_\lam$, $y_k$, $Q_k$ and $Q_k^*$ above \eqref{e:12whitney3}. Then we may write $f_2=\sum_k f\chi_{Q_k}$. So
$$\mathcal{C}_\eps[\nabla A_1,\cdots,\nabla A_n, f_2](x)=\sum_k\mathcal{C}_\eps[\nabla A_1,\cdots,\nabla A_n, f\chi_{Q_k}](x).$$
In the following we study $\prod_{i=1}^n\fr{{A}_i(x)-A_i(y)}{|x-y|}$. We shall separate it into several terms and then give an estimate for each term.  Write
\Bes
\begin{split}
&\ \ \ \ \prod_{i=1}^n\fr{{A}_i(x)-A_i(y)}{|x-y|}\\&=\prod_{i=1}^n\Big(\fr{\tilde{A}_i(x)-\tilde{A}_i(y)}{|x-y|}+\fr{\tilde{A}_i(y)-\tilde{A}_i(y_k)}{|x-y|}+\fr{A_i(y_k)-A_i(y)}{|x-y|}\Big)\\
&=\sum_{\N_1^n}\Big(\prod_{i\in N_1}\fr{\tilde{A}_i(x)-\tilde{A}_i(y)}{|x-y|}\Big)\Big(\prod_{i\in N_2}\fr{\tilde{A}_i(y)-\tilde{A}_i(y_k)}{|x-y|}\Big)\Big(\prod_{i\in N_3}\fr{{A}_i(y_k)-{A}_i(y)}{|x-y|}\Big)\\
&=I(x,y)+II(x,y,y_k)+III(x,y,y_k)+IV(x,y,y_k),
\end{split}
\Ees
where in the third equality we divide $\N_{1}^n=N_1\cup N_2\cup N_3$ with $N_1$, $N_2$, $N_3$ non intersecting each other; and $I(x,y)$, $II(x,y,y_k)$, $III(x,y,y_k)$ and $IV(x,y,y_k)$ are defined as follows
\Be\label{e:12axyqbigd3}
\begin{split}
I(x,y)=&\prod_{i=1}^n\fr{\tilde{A}_i(x)-\tilde{A}_i(y)}{|x-y|},\\
II(x,y,y_k)=&\sum_{N_1\subsetneq\N_{1}^n\atop N_3=\emptyset}\Big(\prod_{i\in N_1}\fr{\tilde{A}_i(x)-\tilde{A}_i(y)}{|x-y|}\Big)\Big(\prod_{i\in N_2}\fr{\tilde{A}_i(y)-\tilde{A}_i(y_k)}{|x-y|}\Big),\\
III(x,y,y_k)=&\sum_{N_1\subsetneq\N_{1}^n\atop N_3\neq\emptyset, N_3\subset\{1,\cdots,l\}}\Big(\prod_{i\in N_1}\fr{\tilde{A}_i(x)-\tilde{A}_i(y)}{|x-y|}\Big)
\Big(\prod_{i\in N_2}\fr{\tilde{A}_i(y)-\tilde{A}_i(y_k)}{|x-y|}\Big)\\
&\qquad\qquad\qquad\times\Big(\prod_{i\in N_3}\fr{{A}_i(y_k)-{A}_i(y)}{|x-y|}\Big),\\
IV(x,y,y_k)=&\sum_{N_1\subsetneq\N_{1}^n\atop N_3 \neq\emptyset, N_3\cap\{l+1,\cdots,n\}\neq\emptyset}\Big(\prod_{i\in N_1}\fr{\tilde{A}_i(x)-\tilde{A}_i(y)}{|x-y|}\Big)
\Big(\prod_{i\in N_2}\fr{\tilde{A}_i(y)-\tilde{A}_i(y_k)}{|x-y|}\Big)\\
&\qquad\qquad\qquad\times\Big(\prod_{i\in N_3}\fr{{A}_i(y_k)-{A}_i(y)}{|x-y|}\Big).
\end{split}
\Ee
In the above decomposition, we in fact divide $\mathcal{C}_\eps[\nabla A_1,\cdots,\nabla A_n, f\chi_{Q_k}](x)$ into $3^n$ terms and separate these terms into four parts according $I$, $II$, $III$ and $IV$.

\subsubsection*{Step 10: Weighted estimate of $\mathcal{C}_*[\cdots,\cdot]$ related to $I$.}
In this case there is only one term, i.e. $\C_*[\nabla\tilde{A}_1,\cdots,\nabla\tilde{A}_n,f_2]$. Then by Proposition \ref{p:12strongr} ($1\leq p<\infty$), we get
\Bes
\begin{split}
w\big(\big\{x\in& (10G_\lam)^c: \mathcal{C}_*[\nabla \tilde{A}_1,\cdots,\nabla\tilde{A}_n,f_2](x)>{\lam}/{2}\big\}\big)\\
&\lc \lam^{-p}\Big(\prod_{i=1}^n\|\nabla\tilde{A}_i\|^p_{L^\infty(\R^d,w)}\Big)\|f_2\|^p_{L^p(\R^d,w)}
\lc\lam^{-p+p\sum_{i=1}^n\fr{r}{q_i}}=\lam^{-r}.
\end{split}
\Ees
If $p=\infty$, the above argument may not work.

\subsubsection*{Step 11: Weighted estimate of $\mathcal{C}_*[\cdots,\cdot]$ related to $II$.} It is sufficient to consider one term $\mathcal{C}_*[\cdots,\cdot]$ related to $II$ in which $N_1$ is a proper subset of $\N_1^n$ and $N_3=\emptyset$.
In such a case, without loss of generality, we may suppose $N_1=\{1,\cdots,v\}$, $N_2=\{v+1,\cdots,n\}$ with $0\leq v<n$. Here if $v=0$, it means that $N_1=\emptyset$. With these notation, we see that $N_1$ is a proper subset of $\N_1^n$.
By a slight abuse of notation, we still utilize $II(x,y,y_k)$ to represent one term related to $N_1$, $N_2$ and $N_3$ in \eqref{e:12axyqbigd3} and utilize $H_{II}(x)$ to represent $\mathcal{C}_*[\cdots,\cdot]$ related to ${II}(x,y,y_k)$, i.e.
$$H_{II}(x)=\sup_{\eps>0}\Big|\sum_k\int_{|x-y|>\eps}K(x-y)II(x,y,y_k)f\chi_{Q_k}(y)dy\Big|.$$
Notice that $\tilde{A}_i$ is a Lipschitz function with bound $\lam^{\fr{r}{q_i}}$ for $i=1,\cdots,n$ by \eqref{e:12qileqde3}. Then we obtain that
\Bes
\begin{split}
|II(x,y,y_k)|\lc\lam^{\sum_{i=1}^n\fr{r}{q_i}}\fr{|y-y_k|^{n-v}}{|x-y|^{n-v}}.
\end{split}
\Ees
Since we only need to consider $x\in(10G_\lam)^c$, then by \eqref{e:12whitney3}, we obtain that
\Be\label{e:12xygeqQ2}
|x-y|\geq 2l(Q_k)\approx|y-y_k|\ \ \text{for any $y\in Q_k$}.
\Ee
Now combining with \eqref{e:12kb}, the above estimate of $II(x,y,y_k)$ and \eqref{e:12xygeqQ2}, we obtain that
\Bes
\begin{split}
H_{II}(x)&\leq\sum_k\int_{Q_k}|K(x-y)|\cdot|II(x,y,y_k)|\cdot|f(y)|dy\\
&\lc\lam^{\sum_{i=1}^n\fr{r}{q_i}}\sum_k\int_{Q_k}\fr{l(Q_k)^{n-v}}{[l(Q_k)+|x-y|]^{d+n-v}}|f(y)|dy.
\end{split}
\Ees

Utilizing the Chebyshev inequality, the above estimate of $H_{II}$ and Lemma \ref{l:12disq1infty} (since $n-v\geq1$), we finally obtain that
\Bes
w(\{x\in(10G_\lam)^c: H_{II}(x)>\lam\})\leq\lam^{-p+\sum_{i=1}^n\fr{rp}{q_i}}\|T_{n-v}f\|^p_{L^p(\R^d,w)}\lc\lam^{-r}\|f\|^p_{L^p(\R^d,w)}.
\Ees
Hence we finish the proof related to $II$.

\subsubsection*{Step 12: Weighted estimate of $\mathcal{C}_*[\cdots,\cdot]$ related to $III$.}
It is sufficient to consider one term $\mathcal{C}_*[\cdots,\cdot]$ related to $III$ in which $N_1$ is a proper subset of $\N_1^n$ and $N_3$ is a nonempty subset of $\{1,\cdots,l\}$. By the condition in this proposition, for any $i\in N_3$, $d\leq q_i<\infty$. Thus $\nabla A_i\in L^{q_i}(\R^d,w)$ (or $L^{d,1}(\R^d,w)$ if $q_i=d$) with $d\leq q_i<\infty$. Then by using the fact $y_k$ lies in the $(G_\lam)^c$, i.e. $y_k\in (J_{i,\lam})^c$, we give the estimates in $N_3$ as follows
\Be\label{e:12mbablaiN3}
\fr{|A_i(y_k)-A_i(y)|}{|y-y_k|}\leq\M(\nabla A_i)(y_k)\leq\lam^{\fr{r}{q_i}}, \ \text{for $i\in N_3$.}
\Ee
Define $v=\card(N_1)$. Then we see that $0\leq v<n$. By a slight abuse of notation, we still utilize $III(x,y,y_k)$ to stand for one term related to $N_1$, $N_2$ and $N_3$ in \eqref{e:12axyqbigd3} and utilize $H_{III}(x)$ to represent $\mathcal{C}_*[\cdots,\cdot]$ related to ${III}(x,y,y_k)$, i.e.
$$H_{III}(x)=\sup_{\eps>0}\Big|\sum_k\int_{|x-y|>\eps}K(x-y)III(x,y,y_k)f\chi_{Q_k}(y)dy\Big|.$$

From the fact $\tilde{A}_i$s are Lipschitz functions with bounds $\lam^{r/q_i}$ for $i\in N_1\cup N_2$ and \eqref{e:12mbablaiN3}, we obtain that
\Bes
|III(x,y,y_k)|\lc\lam^{\sum_{i\in N_{1}\cup N_2}\fr{r}{q_i}}\fr{|y-y_k|^{n-v}}{|x-y|^{n-v}}\prod_{i\in N_3}\M(\nabla A_i)(y_k)\lc\lam^{\sum_{i=1}^n\fr{r}{q_i}}\fr{|y-y_k|^{n-v}}{|x-y|^{n-v}}.
\Ees
Inserting this estimate of $III(x,y,y_k)$ into $H_{III}$, combining with \eqref{e:12kb} and \eqref{e:12xygeqQ2}, and next utilizing the Chebyshev inequality and Lemma \ref{l:12disq1infty} (since $n-v\geq1$), we finally obtain that
\Bes
\begin{split}
w(\{x\in(10G_\lam)^c: H_{III}(x)>\lam\})\leq\lam^{-p+p\big(\sum_{i=1}^n\fr{r}{q_i}\big)}\|T_{n-v}f\|^p_{L^p(\R^d,w)}\lc\lam^{-r}\|f\|^p_{L^p(\R^d,w)}.
\end{split}
\Ees
Hence we finish the proof of this part.

\subsubsection*{Step 13: Weighted estimate of $\mathcal{C}_*[\cdots,\cdot]$ related to $IV$.}
It is sufficient to consider one term $\mathcal{C}_*[\cdots,\cdot]$ related to $IV$ in which $N_1\subsetneq\N_1^n$ and $N_3\neq\emptyset$ with $N_3\cap\{l+1,\cdots,n\}\neq\emptyset$.
In such a case, without loss of generality, we may suppose $l+1,\cdots, v\in N_3$ with $l+1\leq v\leq n$ and $v+1,\cdots,n$ belongs to $N_1$ or $N_2$. So we may assume that $N_3=\{\iota,\cdots,w, l+1,\cdots, v\}$ with $0\leq\iota\leq w\leq l$. Define $u=\card (N_1)$. Then $n-u\geq1$. With these notation, we can easily see that $N_3$ is a nonempty set with $N_3\cap\{l+1,\cdots,n\}\neq\emptyset$. Note that $w\in A_1(\R^d)$.
By a slight abuse of notation, we still use $IV(x,y,y_k)$ to stand for one term related to $N_1$, $N_2$ and $N_3$ in \eqref{e:12axyqbigd3} and use $H_{IV}(x)$ to stand for $\mathcal{C}_*[\cdots,\cdot]$ related to ${IV}(x,y,y_k)$, i.e.
$$H_{IV}(x)=\sup_{\eps>0}\Big|\sum_k\int_{|x-y|>\eps}K(x-y)IV(x,y,y_k)f\chi_{Q_k}(y)dy\Big|.$$

Note that $d\leq q_{1},\cdots,q_l\leq\infty$ and $1\leq q_{l+1},\cdots,q_n<d$.  Recall in Step 3, we set $\fr{1}{s_i}=\fr{1}{q_i}-\fr{1}{d}$ for $i=l+1,\cdots,n$. We also set $\fr{1}{q}=\big(\sum_{i=l+1}^v\fr{1}{s_i}\big)+\fr{1}{p}.$
Since $r\geq \fr{d}{d+n}$ and $\fr{1}{r}=\big(\sum_{i=1}^n\fr{1}{q_i}\big)+\fr{1}{p}$, we could obtain $1\leq q\leq\infty$ which will be crucial when we use Lemma \ref{l:12disq1infty}. With \eqref{e:12whitney3} and $\tilde{A}_i$ is a Lipschitz function with bound $\lam^{r/q_i}$ for $i\in N_1\cup N_2$, we have
\Bes
\begin{split}
|IV(x,y,y_k)|&\lc\lam^{(\sum_{i\in N_1\cup N_2})\fr{r}{q_i}}\fr{(l(Q_k))^{n-u}}{|x-y|^{n-u}}\Big[\prod_{i=\iota}^w\M(\nabla A_i)(y_k)\Big]\prod_{i=l+1}^v\fr{|A_i(y_k)-A_i(y)|}{l(Q_k)}\\
&\lc\lam^{(\sum_{i=1}^l+\sum_{i=v+1}^n)\fr{r}{q_i}}\fr{(l(Q_k))^{n-u}}{|x-y|^{n-u}}\prod_{i=l+1}^v\fr{|A_i(y_k)-A_i(y)|}{l(Q_k)}.
\end{split}
\Ees
Then inserting the above estimate of $IV$ into $H_{IV}$ with \eqref{e:12kb} and \eqref{e:12xygeqQ2}, we get
\Bes
\begin{split}
H_{IV}(x)&\leq\sum_k\int_{Q_k}|K(x-y)|\cdot|IV(x,y,y_k)|\cdot|f(y)|dy\\
&\lc\lam^{(\sum_{i=1}^l+\sum_{i=v+1}^n)\fr{r}{q_i}}\sum_k\int_{Q_k}\fr{l(Q_k)^{n-u}}{[l(Q_k)+|x-y|]^{d+n-u}}h_{l,v}(y)dy\\
&=\lam^{(\sum_{i=1}^l+\sum_{i=v+1}^n)\fr{r}{q_i}}T_{n-u}\big(h_{l,v}\big)(x),
\end{split}
\Ees
where the operator $T_{n-u}$ is defined in Lemma \ref{l:12disq1infty} and the function $h_{l,v}(y)$ is defined as $$h_{l,v}(y)=\sum_{Q_k}\prod_{i=l+1}^v\Big(\fr{|A_i(y_k)-A_i(y)|}{l(Q_k)}\Big)\chi_{Q_k}|f|(y).$$

Utilizing the Chebyshev inequality and the above estimate of $H_{IV}$, applying  Lemma \ref{l:12disq1infty}(note that $1\leq q\leq\infty$ and $n-u\geq1$), we finally obtain that
\Be\label{e:12Glam3qleqd3}
\begin{split}
w(\{x\in&(10G_\lam)^c: H_{IV}(x)>\lam\})\\
&\leq\lam^{-q+(\sum_{i=1}^l+\sum_{i=v+1}^n)\fr{rq}{q_i}}\int_{(10G_\lam)^c}[T_{n-u}(h_{l,v})(x)]^qw(x)dx\\
&\lc\lam^{-q+(\sum_{i=1}^l+\sum_{i=v+1}^n)\fr{rq}{q_i}}\|h_{l,v}\|^q_{L^q(\R^d,w)}.
\end{split}
\Ee
In the following we give an estimate of $\|h_{l,v}\|^q_{L^q(\R^d,w)}$. We may write
\Bes
\begin{split}
\|h_{l,v}\|^q_{L^q(\R^d,w)}&=\sum_{Q_k}\int_{Q_k}\Big[\prod_{i=l+1}^v\Big(\fr{|A_i(y_k)-A_i(y)|}{l(Q_k)}\Big)^q\Big]|f(y)|^qw(y)dy\\
&\leq\sum_{Q_k}\prod_{i=l+1}^v\Big[\int_{Q_k}\Big(\fr{|A_i(y_k)-A_i(y)|}{l(Q_k)}\Big)^{s_i}w(y)dy\Big]^{\fr{q}{s_i}}\Big[\int_{Q_k}|f(y)|^pw(y)dy\Big]^{\fr{q}{p}}\\
&\lc\sum_{Q_k}\prod_{i=l+1}^v\Big[\int_{Q^*_k}\Big(\fr{|A_i(y_k)-A_i(y)|}{l(Q^*_k)}\Big)^{s_i}w(y)dy\Big]^{\fr{q}{s_i}}\Big[\int_{Q_k}|f(y)|^pw(y)dy\Big]^{\fr{q}{p}}\\
&\lc\sum_{Q_k}\Big[\prod_{i=l+1}^v\mathfrak{M}_{w,s_i}(\nabla A_i)(y_k)^{q}w(Q_k)^{\fr{q}{s_i}}\Big]\Big[\int_{Q_k}|f(y)|^pw(y)dy\Big]^{\fr{q}{p}},
\end{split}
\Ees
where in the second inequality we use the H\"older inequality and the third inequality follows from the fact $Q_k\subset Q_k^*$, $y_k$ is the center of $Q^*_k$ and $l(Q^*_k)\approx l(Q_k)$. Notice that $y_k$ lies in the $(G_\lam)^c$, i.e. $y_k\in (D_{i,\lam})^c$ (see Step 3). Then we obtain that
\Bes
\mathfrak{M}_{w,s_i}(\nabla A_i)(y_k)\leq\lam^{\fr{r}{q_i}}, \ \text{for $i=l+1,\cdots,v$.}
\Ees
Utilizing the above inequality, the H\"older inequality again and (cz-iii) in Step 2, we get
\Bes
\begin{split}
\|h_{l,v}\|^q_{L^q(\R^d,w)}&\lc\lam^{\sum_{i=l+1}^v\fr{qr}{q_i}}\Big[\sum_{Q_k}w(Q_k)\Big]^{\sum_{i=l+1}^v\fr{q}{s_i}}\|f\|^q_{L^p(\R^d,w)}\\
&\lc\lam^{\sum_{i=l+1}^v\fr{qr}{q_i}}\big[w(G_\lam)\big]^{\sum_{i=l+1}^v\fr{q}{s_i}}\lc\lam^{\sum_{i=l+1}^v\Big(\fr{qr}{q_i}-\fr{qr}{s_i}\Big)}.
\end{split}
\Ees
Plunge the above estimate into \eqref{e:12Glam3qleqd3} with some elementary calculations, we finally obtain that
\Bes
\begin{split}
w(\{x\in(10G_\lam)^c: H_{IV}(x)>\lam\})\lc\lam^{-q+\big(\sum_{i=1}^n\fr{rq}{q_i}\big)-\big(\sum_{i=l+1}^v\fr{qr}{s_i}\big)}\lc\lam^{-r},
\end{split}
\Ees
hence we finish the proof of the term $IV$.

Finally, we show how to modify the above argument to the case $q_i=\infty$ for some $i=k+1,\cdots,l$.  Notice that only in Step 1 the construction of {\it exceptional set} is involved with $A_{k+1},\cdots, A_{l}$. We may assume that only $q_{k+1}=\cdots=q_u=\infty$ with $k+1\leq u\leq l$. Therefore $A_{k+1}$, $\cdots$, $A_u$ are Lipschitz functions. Then we just fix $A_{k+1}, \cdots, A_u$ in the rest of the proof. In Step 1 we modify the argument that we only make a construction of {\it exceptional set\/} for $A_1,\cdots, A_k$ and $A_{u+1}, \cdots, A_l$. These proofs in Steps 2-8 are the same. Later when studying $$\Big(\prod_{i={1}}^n\fr{A_i(x)-A_i(y)}{|x-y|}\Big)=\Big(\prod_{i={k+1}}^u\prod_{i={1}}^k\prod_{i={u+1}}^n\Big)\fr{A_i(x)-A_i(y)}{|x-y|},$$
we just use the same way as in Steps 9-13 to deal with the terms  from $\prod_{i={1}}^k\prod_{i={u+1}}^n$
since the term $\prod_{i={k+1}}^u\fr{A_i(x)-A_i(y)}{|x-y|}$ could be absorbed by the kernel $K(x-y)$ if we observe that $K(x-y)\prod_{i={k+1}}^u\fr{A_i(x)-A_i(y)}{|x-y|}$
is a standard Calder\'on-Zygmund kernel.
\end{proof}

\begin{prop}\label{p:12qgeqleqdinfty}
Let $\fr{d}{d+n}\leq r\leq 1$, $q_{1}=\cdots=q_{k}=d$, $d< q_{k+1},\cdots,q_l\leq\infty$ and $1\leq q_{l+1},\cdots,q_n< d$ with $0\leq k\leq l$ and $1\leq l<n$, $p=\infty$. Suppose that $w\in A_1(\R^d)$. Then
\Bes
\begin{split}
\|\C_*[\nabla A_1,&\cdots,\nabla A_n, f]\|_{L^{r,\infty}(\R^d,w)}\\
&\lc\Big(\prod_{i=1}^k\|\nabla A_i\|_{L^{d,1}(\R^d,w)}\Big)\Big(\prod_{i=k+1}^n\|\nabla A_i\|_{L^{q_i}(\R^d,w)}\Big)\|f\|_{L^\infty(\R^d,w)},
\end{split}
\Ees
where $L^{d,1}(\R^d,w)$ is the standard Lorentz space.
\end{prop}
\begin{proof}
The proof is similar to that of Proposition \ref{p:12qgeqleqd} and one could follow the idea in the proof of Proposition \ref{p:12qleqinfty}, so the details of the proof is omitted.
\end{proof}
\vskip0.24cm
\subsection{Interpolation}\label{s:1226}\quad
\vskip0.24cm
Notice that we have already proven all the cases (ii) in Theorem \ref{t:12} by Propositions \ref{p:12qibigd},\ref{p:12qleqinfty},\ref{p:12qgeqleqd} and \ref{p:12qgeqleqdinfty}. And only part strong type multilinear estimates of (i) in Theorem \ref{t:12} has been established by Proposition \ref{p:12strongr} and Proposition \ref{p:16maxinf}. The rest part of (i) in Theorem \ref{t:12} just follow from the linear Marcinkiewicz  interpolation (see \cite{Ste71} or \cite{BL76}). In the following we show how to do this.

Since the maximal Calder\'on commutator $\C_*$ is $(n+1)$th submultilinear, when using the Marcinkiewicz interpolation, our main strategy is that we consider $\C_*$ as a sublinear operator if we fix part of $n$ variables.

Let $\nabla A_i\in L^{q_i}(\R^d,w)$ and $f\in L^p(\R^d,w)$ with $\fr{1}{r}=\big(\sum_{i=1}^n\fr{1}{q_i}\big)+\fr{1}{p}$, $\fr{d}{d+n}<r<\infty$, $1< q_i\leq\infty$ $(i=1,\cdots,n)$ and $1< p\leq\infty$. Let $w\in \big(\bigcap_{i=1}^nA_{\max\{\fr{q_i}{d},1\}}(\R^d)\big)\cap {A_p(\R^d)}$. Our goal is to show the follow strong type estimate
\Be\label{e:16strong}
\|\C_*[\nabla A_1,\cdots,\nabla A_n, f]\|_{L^r(\R^d,w)}\lc\Big(\prod_{i=1}^n\|\nabla A_i\|_{L^{q_i}(\R^d,w)}\Big)\|f\|_{L^p(\R^d,w)}.
\Ee

We divide the proof into several cases. We first consider the case all $q_i\neq d$ for $i=1,\cdots,n$. Therefore by (ii) of Theorem \ref{t:12}, the multilinear estimates \eqref{e:16mwek} are not involved with $L^{d,1}(\R^d,w)$ spaces. We further divide this case into two cases: $1<p<\infty$ and $p=\infty$. Consider firstly the case $1<p<\infty$. We fix all $\nabla A_i$, $q_i$ and $w\in \big(\bigcap_{i=1}^nA_{\max\{\fr{q_i}{d},1\}}(\R^d)\big)\cap {A_p}(\R^d)$. By the basic property of $A_p(\R^d)$ weight, $w\in A_{p_1}(\R^d)$ for all $p_1>p$. If we choose $p_1,r_1$ such that $p<p_1<\infty$ and $\fr{1}{r_1}=\big(\sum_{i=1}^n\fr{1}{q_i}\big)+\fr{1}{p_1}$, then by (ii) of Theorem \ref{t:12}
\Be\label{e:16p1}
\C_*[\nabla A_1,\cdots,\nabla A_n,\cdot]: L^{p_1}(\R^d,w)\mapsto L^{r_1,\infty}(\R^d,w).
\Ee
Since $w\in A_{p}(\R^d)$, by the revers H\"older inequality of $A_{p}(\R^d)$ weight (see \cite{Gra249}) and its definition, there exist $\eps>0$ such that $w\in A_{p-\eps}$ and $p-\eps\geq1$. Then we may choose $p_0,r_0$ such that $p-\eps\leq p_0<p$, $\fr{d}{d+n}<r_0<\infty$ and $\fr{1}{r_0}=\big(\sum_{i=1}^n\fr{1}{q_i}\big)+\fr{1}{p_0}$. Hence we obtain $w\in \big(\bigcap_{i=1}^nA_{\max\{\fr{q_i}{d},1\}}(\R^d)\big)\cap {A_{p_0}(\R^d)}$. By using (ii) of Theorem \ref{t:12}, we get
\Be\label{e:16p0}
\C_*[\nabla A_1,\cdots,\nabla A_n,\cdot]: L^{p_0}(\R^d,w)\mapsto L^{r_0,\infty}(\R^d,w).
\Ee
Applying the Marcinkiewicz interpolation with \eqref{e:16p1} and \eqref{e:16p0}, we establish the strong type estimate \eqref{e:16strong} provided that all $q_i\neq d$ and $1<p<\infty$.
Next we consider another case all $q_i\neq d$ and $p=\infty$. By our condition $r<\infty$, there is at least one $q_i<\infty$. Without loss of generality, we may suppose that $q_1<\infty$. If $d<q_1<\infty$, then the rest of proof is similar to the case $q_i\neq d$ and $1<p<\infty$ once fixing  $\nabla A_i$, $q_i$ for $i=2,\cdots,n$, $f\in L^\infty(\R^d,w)$ and $w\in \bigcap_{i=2}^nA_{\max\{\fr{q_i}{d},1\}}(\R^d)$. If $1<q_1<d$, then $w\in A_1(\R^d)$ by our condition. Therefore it is easy to show \eqref{e:16strong} using (ii) of Theorem \ref{t:12} once we fix $\nabla A_i$, $q_i$ for $i=2,\cdots,n$, $f\in L^\infty(\R^d,w)$ and $w\in A_1(\R^d)$.

Secondly let us consider the case there is only one $q_i$ which equals to $d$. Without loss of generality, we may suppose $q_1=d$. Then by our condition $w\in A_1(\R^d)$ in this case. Fix $\nabla A_i$, $q_i$ for $i=2,\cdots,n$, $f\in L^p(\R^d,w)$ and $w\in A_1(\R^d)$. Then we may choose $r_0,r_1, q_{1,0}, q_{1,1}$ such that $\fr{d}{d+n}<r_0,r_1<\infty$, $1<q_{1,0}<d<q_{1,1}$, $\fr{1}{r_0}=\fr{1}{q_{1,0}}+\big(\sum_{i=2}^n\fr{1}{q_i}\big)+\fr{1}{p}$, $\fr{1}{r_1}=\fr{1}{q_{1,1}}+\big(\sum_{i=2}^n\fr{1}{q_i}\big)+\fr{1}{p}$. Then by (ii) of Theorem \ref{t:12}, we get
\Bes
\C_*[\cdot,\nabla A_2,\cdots,\nabla A_n,f]: L^{q_{1,j}}(\R^d,w)\mapsto L^{r_j,\infty}(\R^d,w)\quad j=0,1.
\Ees
Using the Marcinkiewicz interpolation with the above two estimate, we get \eqref{e:16strong} in the case $q_1=d$ and all $q_2,\cdots,q_n\neq d$.

Finally we consider the general case there are $m$ numbers of $q_i$s which equal to $d$. We only need to show $m=2$, the general case just follows from the induction. Without loss of generality, we suppose that $q_1=q_2=d$. In this case, $w\in A_1(\R^d)$. Fix $\nabla A_i$, $q_i$ for $i=2,\cdots,n$, $f\in L^p(\R^d,w)$ and $w\in A_1(\R^d)$. Then we may choose $r_0,r_1, q_{1,0}, q_{1,1}$ such that $\fr{d}{d+n}<r_0,r_1<\infty$, $1<q_{1,0}<d<q_{1,1}$, $\fr{1}{r_0}=\fr{1}{q_{1,0}}+\big(\sum_{i=2}^n\fr{1}{q_i}\big)+\fr{1}{p}$, $\fr{1}{r_1}=\fr{1}{q_{1,1}}+\big(\sum_{i=2}^n\fr{1}{q_i}\big)+\fr{1}{p}$. Since $q_2=d$, by the result of the case there is only one $q_i=d$ we discussed above, we get the strong type estimate
\Bes
\C_*[\cdot,\nabla A_2,\cdots,\nabla A_n,f]: L^{q_{1,j}}(\R^d,w)\mapsto L^{r_j}(\R^d,w)\quad j=0,1.
\Ees
Using the Marcinkiewicz interpolation with the above two estimate, we get \eqref{e:16strong} in the case $q_1=q_2=d$ and all $q_3,\cdots,q_n\neq d$. Applying the induction of $m$, we finish the proof.

\begin{remark}
Instead of using the linear Marcinkiewicz interpolation in this proof, another possible more straightforward method is the multilinear interpolation with change of measures. To the best knowledge of the author, such kind of multilinear interpolation with change of measures is currently unknown. Therefore it will be interesting to establish the multilinear version of Stein-Weiss interpolation with change of measures (see \cite{SW58}).
\end{remark}

\subsection*{Acknowledgement} The author would like to thank the referees for their very careful reading and valuable suggestions.

\bibliographystyle{amsplain}

\begin{thebibliography}{10}


\bibitem{BL76}J. Bergh and J. L\"ofstr\"om, \textit{Interpolation spaces. An introduction.} Grundlehren der Mathematischen Wissenschaften, No. 223. Springer-Verlag, Berlin-New York, 1976. x+207 pp.

\bibitem {Cal65}A. P. Calder\'on, \textit{Commutators of singular integral operators}, Proc. Nat. Acd. Sci. USA. \textbf{53} (1965), 1092-1099.

\bibitem{Cal77} A. P. Calder\'on, \textit{Cauchy integrals on Lipschitz curves and related operators,} Proc. Nat. Acad. Sci. USA, \textbf{74} (1977), 1324-1327.

\bibitem{Cal78} A. P. Calder\'on, \textit{Commutators, singular integrals on Lipschitz curves and application}, Proc. Inter. Con. Math. Helsinki, 1978, 85-96, Acad. Sci. Fennica, Helsinki, 1980.

\bibitem {CCal75} C. P. Calder\'on, \textit{On commutators of singular integrals}, Studia Math., \textbf{53} (1975), 139-174.


\bibitem{CM75} R. Coifman and Y. Meyer, \textit{On commutators of singular integral and bilinear singular integrals}, Trans. Amer. Math. Soc., 212 (1975), 315-331.

\bibitem{CM78} R. Coifman and Y. Meyer, \textit{Au del\`a des op\'erateurs pseudo-diff¨¦rentiels}, Ast\'{e}risque, 57 (1978).

\bibitem {CJ87}M. Christ and J. L. Journ\'e, \textit{Polynomial growth estimates for multilinear singular integral operators}, Acta Math. \textbf{159} (1987), 51-80.

\bibitem{DS90}G. David and S. Semmes, \textit{Strong $A_\infty$ weights, Sobolev inequalities and quasiconformal mappings.} Analysis and partial differential equations, 101-111, Lecture Notes in Pure and Appl. Math., 122, Dekker, New York, 1990.

\bibitem{DL18}Y. Ding and X. Lai, \textit{Weak type (1,1) bound criterion for singular integral with rough kernel and its applications}. Trans. Amer. Math. Soc. 371(2019), No. 3, 1649-1675.

\bibitem{DGY10}X. Duong, L. Grafakos and L. Yan,  \textit{Multilinear operators with non-smooth kernels and commutators of singular integrals}. Trans. Amer. Math. Soc. 362 (2010), no. 4, 2089-2113.

\bibitem{DGGLY09}X. Duong, R. Gong, L. Grafakos, J. Li and L. Yan, \textit{Maximal operator for multilinear singular integrals with non-smooth kernels}. Indiana Univ. Math. J. 58 (2009), no. 6, 2517-2541.

\bibitem{Fef74} C. Fefferman, \textit{Recent Progress in Classical Fourier Analysis}, Proc. Inter. Con. Math., Vancouver, 	1974, 95-118.

\bibitem{Fon16}P. W. Fong, \textit{Smoothness properties of symbols, Calder\'on commutators and generalizations}. Thesis (Ph.D.)-Cornell University. 2016. 60 pp.
    
\bibitem{GR85}J. Garc\'{\i}a-Cuerva and J. Rubio de Francia, \textit{Weighted norm inequalities and related topics}, North-Holland Math. Studies \textbf{116}, North-Holland, Amsterdam, 1985.

\bibitem{GLY11}L. Grafakos, L. Liu and D. Yang, \textit{Multiple-weighted norm inequalities for maximal multi-linear singular integrals with non-smooth kernels.} Proc. Roy. Soc. Edinburgh Sect. A 141 (2011), no. 4, 755-775.

\bibitem {Gra249}L. Grafakos, \textit{Classic Fourier Analysis}, Graduate Texts in Mathematics, Vol. \textbf{249} (Third edition), Springer, New York, 2014.

\bibitem {Gra250}L. Grafakos, \textit{Modern Fourier Analysis}, Graduate Texts in Mathematics, Vol. \textbf{250} (Third edition), Springer, New York, 2014.

\bibitem{HSSS17}M. Had\v{z}i\'{c}, A. Seeger, C. K. Smart, B. Street, \textit{Singular integrals and a problem on mixing flows.} Ann. Inst. H. Poincare Anal. Non Lineaire. 35 (2018), no. 4, 921-43.

\bibitem{Jou83}J. L. Journ\'e, \textit{Calder\'on-Zygmund operators, pseudodifferential operators and the Cauchy integral of Calder\'on.} Lecture Notes in Mathematics, 994. Springer-Verlag, Berlin, 1983. vi+128 pp.


\bibitem{Lai17}X. Lai, \textit{Multilinear estimates for Calder\'on commutators}. Int. Math. Res. Not. IMRN. 2018. Doi: 10.1093/imrn/rny197. See also arXiv:1712.09020.

\bibitem{LOPTT09}A. K. Lerner, S. Ombrosi, C. P\'erez, R. H. Torres, and R. Trujillo-Gonz¨¢lez, \textit{New maximal functions and multiple weights for the multilinear Calder\'on-Zygmund theory.} Adv. Math. 220 (2009), no. 4, 1222-1264.

\bibitem{Leg18}F. L\'eger, \textit{A new approach to bounds on mixing.} Math. Models Methods Appl. Sci. \textbf{28} (2018), no. 5, 829-849.

\bibitem {MC97} Y. Meyer and R. Coifman, \textit{Wavelets. Calder\'on-Zygmund and multilinear operators}. Translated from the 1990 and 1991 French originals by David Salinger. Cambridge Studies in Advanced Mathematics, 48. Cambridge University Press, Cambridge, 1997.

\bibitem {SSS15}A. Seeger, C. K. Smart and B. Street, \textit{Multilinear singular integral forms of Christ-Journ\'e type.} Mem. Amer. Math. Soc. vol. 257, No. 1231 (2019).

\bibitem {Ste70}E. M. Stein, \textit{Singular Integrals and Differentiability Properties of Functions.}  Princeton Mathematical Series, No. 30, Princeton University Press, Princeton, N.J. 1970.

\bibitem{SW58}E. M. Stein and G. Weiss, \textit{Interpolation of operators with change of measures.} Trans. Amer. Math. Soc. 87 (1958), 159-172.

\bibitem {Ste71}E. M. Stein and G. Weiss, \textit{Introduction to Fourier Analysis on Euclidean Spaces.}  Princeton Mathematical Series, No. 32, Princeton University Press, Princeton, N.J. 1971.


\end{thebibliography}

\end{document}